\documentclass[12pt]{article}
\usepackage{e-jc}

\usepackage{amsthm,amsmath,amssymb,amsfonts,amscd}

\usepackage{tikz}
\usepackage{tkz-tab}
\usepackage{ytableau}

\usepackage{geometry}
\usepackage{multirow}

\allowdisplaybreaks[1]

\usepackage{graphicx}

\usepackage{graphics}

\usepackage{mathtools}

\usepackage{subcaption}

\usepackage[colorlinks=true,citecolor=black,linkcolor=black,urlcolor=blue]{hyperref}

\newcommand{\arxiv}[1]{\href{http://arxiv.org/abs/#1}{\texttt{arXiv:#1}}}

\newcommand\myeq{\mathrel{\overset{\makebox[0pt]{\mbox{\normalfont\tiny\sffamily d}}}{=}}}
\newcommand{\taua}{\tau^{(1)}}
\newcommand{\taub}{\tau^{(2)}}
\newcommand{\tauc}{\tau^{(3)}}

\DeclareMathOperator{\SD}{\mathbb{SD}}
\DeclareMathOperator{\E}{\mathbb{E}}

\DeclareMathOperator{\bP}{\mathbb P}
\DeclareMathOperator{\LIS}{LIS}
\DeclareMathOperator{\A}{A}
\DeclareMathOperator{\W}{W}
\DeclareMathOperator{\LDS}{LDS}
\DeclareMathOperator{\LAS}{LAS}
\DeclareMathOperator{\Bin}{Bin}
\DeclareMathOperator{\LR}{LR}
\DeclareMathOperator{\RL}{RL}
\DeclareMathOperator{\D}{D}
\DeclareMathOperator{\R}{R}
\DeclareMathOperator{\T}{T}
\DeclarePairedDelimiter\ceil{\lceil}{\rceil}
\DeclarePairedDelimiter\floor{\lfloor}{\rfloor}

\theoremstyle{plain}
\newtheorem{theorem}{Theorem}
\newtheorem{lemma}[theorem]{Lemma}
\newtheorem{corollary}[theorem]{Corollary}

\theoremstyle{definition}

\newtheorem{example}[theorem]{Example}

\theoremstyle{remark}
\newtheorem{remark}[theorem]{Remark}

\title{\bf Longest monotone subsequences and rare regions of pattern-avoiding permutations}

\author{Neal Madras{\thanks{Supported in part by a Discovery Grant from NSERC of Canada.}}\, and G\"{o}khan Y{\i}ld{\i}r{\i}m\\
\small Department of Mathematics and Statistics\\[-0.8ex]
\small York University\\[-0.8ex] 
\small Toronto, Ontario M3J 1P3, Canada\\
\small\tt madras@mathstat.yorku.ca \\
\small\tt gyildir@yorku.ca
}

\begin{document}

\maketitle


\begin{abstract}
We consider the distributions of the lengths of the longest monotone and alternating subsequences
 in classes of permutations of size $n$ that avoid a specific pattern or set of patterns, with
respect to the uniform distribution on each such class.
We obtain exact results for any class that avoids two patterns of length 3, as well as
results for some classes that avoid one pattern of length 4 or more.  
In our results, the longest 
 monotone subsequences have expected length proportional to $n$ for pattern-avoiding classes, 
 in contrast with the $\sqrt n$ behaviour that holds for unrestricted permutations.

In addition, for a pattern $\tau$ of length $k$, we scale the plot of a random $\tau$-avoiding permutation 
down to the unit square and study the ``rare region,'' which is the part of the square that is
 exponentially unlikely to contain any points. 
 We prove that when $\tau_1>\tau_k$, the complement of the rare region is a closed set 
 that contains the main diagonal of the unit square. For the case $\tau_1=k,$ we also show that the lower boundary of the part of the rare region above the main diagonal is a curve that is Lipschitz continuous and strictly increasing on $[0,1]$.

\end{abstract}

\section{Introduction}


For each integer $n\geq 1,$ let $[n]:=\{1,2,\cdots,n\}.$ A permutation of $[n]$ is a bijection 
$\sigma : [n]\to[n],$ written as $\sigma=\sigma_1\sigma_2\cdots\sigma_n$ where 
$\sigma_i=\sigma(i) $ for $i \in [n]$. The set of all permutations on $[n]$ is denoted by $S_n$. 
For $\tau\in S_k$, we write $S_n(\tau)$ to denote the set of all permutations in $S_n$ that
avoid the pattern $\tau$.   (See Section \ref{sec-not} for complete definitions.)

We use $\bP^{\tau}_n$ to denote the uniform probability measure on $S_n(\tau):$ for any subset 
$A$ of $S_n(\tau),$ we have $\bP^{\tau}_n(A):=\frac{|A|}{|S_n(\tau)|}.$ 
We use $\E^{\tau}_n(X)$ and $\SD^{\tau}_n(X)$ to denote the expected value and standard 
deviation of a random variable $X$ on $S_n(\tau)$ under $\bP^{\tau}_n$. 
More generally, if $T$ is a set of patterns, then $S_n(T)$ denotes the set of permutations in 
$S_n$ avoiding all the patterns in $T$,  and $\bP_n^T$, $\E^T_n$ and $\SD^{T}_n$ are the
corresponding probability operators.

The cardinalities $|S_n(T)|$ have been studied extensively by researchers in the last few
decades but have been computed only for some limited cases of sets $T$.

It took 24 years to prove the 1980 conjecture of Stanley and Wilf, which  says that
\begin{equation}
L(\tau):=\lim_{n\to \infty}|S_n(\tau)|^{1/n}\quad \text{exists and is finite for every }\tau \in S_k
\end{equation}
(existence was proved by Arratia \cite{A}, and finiteness by Marcus and Tardos~\cite{MT}). 

For more on pattern-avoiding permutations, see chapters 4 and 5 in \cite{Bo}, \cite{Kitaev}, and the survey paper \cite{Vatter}.

For a permutation $\sigma=\sigma_1 \sigma_2 \cdots \sigma_n \in S_n,$ the \textit{complement} of 
$\sigma$ is $\sigma^c=\sigma^c_1 \sigma^c_2 \cdots \sigma^c_n$ where 
$\sigma^c_i=n+1-\sigma_i$ for $i\in [n].$ The \textit{reverse} of $\sigma$ is defined to be 
$\sigma^r=\sigma_n \sigma_{n-1} \cdots \sigma_1.$ 
These operations give rise to bijections among $S_n(\tau), S_n(\tau^c)$ and $S_n(\tau^r)$.
In addition, inversion gives a bijection from $S_n(\tau)$ to $S_n(\tau^{-1})$.
Therefore,
\[|S_n(\tau)|=|S_n(\tau^c)|=|S_n(\tau^r)|=|S_n(\tau^{-1})|  \,.  \]

\bigskip

The  paper is organized as follows.   Sections \ref{sec-intmono}, \ref{sec-intalt}, and 
\ref{sec-intrare} present the introduction, background, and overview of our results on the
topics of
longest monotone subsequences, longest alternating sequences, and rare regions
respectively.   Section \ref{sec-not} lists some basic notation and terminology.
In sections \ref{sec.tau1tau2} and \ref{LASsection}, we will state and prove our results related to the distribution of longest monotone and alternating subsequences respectively of random permutations from $S_n(\taua,\taub)$ where $\taua,\taub\in S_3$. 
Section \ref{longerpattern} contains some results on patterns of length greater than three. 
Finally, in section \ref{rsection}, we will present our results on the rare regions in 
permutations' plots.

\subsection{Longest monotone subsequences}
    \label{sec-intmono}
For a given $\sigma \in S_n,$ we say that $\sigma_{i_1}\sigma_{i_2}\cdots\sigma_{i_k}$ is an \textit{increasing subsequence of length} $k$ in $\sigma$ if $i_1<i_2<\cdots<i_k$ and $\sigma_{i_1}<\sigma_{i_2}<\cdots<\sigma_{i_k}$. Let $\LIS_n(\sigma)$ be the length of the longest increasing subsequence in $\sigma$. Similarly, let $\LDS_n(\sigma)$ be the length of the longest decreasing subsequence in $\sigma$. 
In a short and elegant argument, usually called the \textit{Erd\"{o}s-Szekeres lemma}, Erd\"{o}s and Szekeres \cite{ES} proved that 
every permutation of length $(r{-}1)(s{-}1)+1$ or more contains either a decreasing subsequence of length $r$ or an increasing subsequence of length $s$; equivalently, 
$\LIS_n(\sigma)\cdot \LDS_n(\sigma)\geq n$ for every $\sigma\in S_n$.

Determining the asymptotic distribution of $\LIS_n$ on $S_n$ under the uniform distribution has a 
rich and interesting history \cite{AD}. The efforts of many researchers around this problem 
culminated in the celebrated result of Baik, Deift and Johansson in 1999 \cite{BDJ} which completely
determined the asymptotic distribution of $\LIS_n.$ They proved that
 \[\lim_{n\to \infty}\bP\left(\frac{\LIS_n-2\sqrt n}{n^{1/6}}\leq t\right)=F(t) \quad \text{for all} \quad t\in \mathbb{R},\]
where  $F$ is the Tracy-Widom distribution function.
 This distribution was first obtained by Tracy and Widom~\cite{TW} in the context of random matrix theory as the distributional limit of the (centered and scaled) largest eigenvalue of the Gaussian unitary ensemble.

 The longest increasing subsequence problem in the context of pattern-avoiding permutations was first studied by Deutsch, Hildebrand and Wilf \cite{DHW} for $S_n(\tau)$ where $\tau \in S_3$. 
 They showed that in the case $\tau=231$, $\E^{231}_n(\LIS_n)\sim \frac{n}{2}$, 
 $\SD^{231}_n(\LIS_n)\sim\frac{\sqrt n}{2}$ and the normalized $\LIS_n$ converges to the standard 
 normal distribution. For $\tau=132$ and $\tau=321$, we have $\E^{132}_n(\LIS_n)\sim \sqrt{\pi n}$, 
 $\SD^{132}_n(\LIS_n)\sim\sqrt n$, and $\E^{321}_n(\LIS_n)\sim \frac{n}{2}$, 
 $ \SD^{321}_n(\LIS_n)\sim\sqrt n$;  for both of them the normalized $\LIS_n$ converges to 
 non-normal distributions which are evaluated exactly in \cite{DHW}. 
 Since $\LIS_n(\sigma)=\LDS_n(\sigma^r)=\LIS_n(\sigma^{rc})=\LIS_n(\sigma^{-1}),$ their results 
 give a complete picture for $\LIS_n$ and $\LDS_n$ on  $S_n(\tau)$ under the uniform measure for 
 every $\tau\in S_3$.
 We are not aware of any existing work on this problem for other classes of pattern-avoiding permutations.
 
 In section~\ref{sec.tau1tau2}, we determine the distributions of $\LIS_n$ and $\LDS_n$ on $S_n(T)$ for $T \subset S_3$ with $|T|=2$.   The results are summarized in Table \ref{summary}.  
The operations \textit{reverse, complement,} and \textit{inverse} induce obvious symmetry 
bijections among the classes  $S_n(T)$ for different pairs $T$, resulting in the five groupings
shown in  Figure~\ref{graph}.  Group (e) is empty for large $n$, so we shall present our results
in terms of groupings (a) thorugh (d).

\begin{table}[h!]
\begin{center}
\begin{tabular}{ ||c|c|c|c||c|c|c|| } 
 \hline
 &\multicolumn{3}{|c||}{$\LIS_n$}&\multicolumn{3}{|c||}{$\LDS_n$} \\
 \hline
 \hline
 &&&&&&\\
 $\{\taua,\taub\}$ & mean & standard&asymp.&mean&standard&asymp.\\ 
 &&deviation&normal?&&deviation&normal?\\
 \hline
 \hline
 &&&&&&\\
 (a) $\{132,321\}$ & $\sim\frac{5n}{6}$& $\sim\frac{5n}{6\sqrt 2}$&No &$\to 2$&$\to 0$&No \\
 &&&&&&\\
 (b) $\{132,231\}$ &  $\frac{n+1}{2}$ & $\frac{\sqrt {n-1}}{2}$&Yes&  $\frac{n+1}{2}$ & $\frac{\sqrt {n-1}}{2}$&Yes\\
 &&&&&&\\
 (c) $\{132,123\}$ &$\to 2$&$\to 0$ & No& $\frac{3n}{4}$& $\sim\frac{\sqrt n}{4}$&Yes\\
 &&&&&&\\
(d) $\{132,213\}$ & $\sim \log_2n$&$\to \text{constant}$&No&$\frac{n+1}{2}$ & $\frac{\sqrt {n-1}}{2}$&Yes\\
 &&&&&&\\
 \hline
\end{tabular}
\end{center}
\caption{
Summary of asymptotic behaviour of $\LIS_n$ and $\LDS_n$ on $S_n(\taua,\taub)$
for the groupings (a)--(d) of Figure \ref{graph}.  See Theorem~\ref{MainTh1} and Figure~\ref{lastfigure}.    
(We write ``$\rightarrow$'' to mean ``converges to,'' and ``$a_n\sim b_n$'' to mean 
``$a_n/b_n\rightarrow 1$''.  Otherwise, expressions are exact.)
}
\label{summary}
\end{table}

\begin{figure}
    
     \begin{subfigure}[b]{0.30\textwidth}
\centering
          \begin{tikzpicture}[node distance=4cm]
\node (A) at (0, 0) {$(312,123)$};
\node (B) at (3, 0) {(213,321)};
\node (C) at (3,3) {(231,123)};
\node (D) at (0,3) {(132,321)};


\draw [<->] (A) edge (B) (B) edge (C) (C) edge (D) (A) edge (D) (C) edge (A);
\draw [->](-0.2,3.3) .. controls (0,4)  .. (0.2,3.3);
\draw [->](2.8,-0.3) .. controls (3,-1)  .. (3.2,-0.3);
\put(40,2){{\footnotesize $r$}}
\put(40,90){{\footnotesize $r$}}
\put(40,50){{\footnotesize $i$}}
\put(90,50){{\footnotesize $c$}}
\put(-10,50){{\footnotesize $c$}}
\put(0,110){{\footnotesize $i$}}
\put(90,-20){{\footnotesize $i$}}
\end{tikzpicture}
\caption{}
          \end{subfigure}
          \begin{subfigure}[b]{0.30\textwidth}
\centering
          \begin{tikzpicture}[node distance=4cm]

\node (A) at (0, 0) {(132,312)};
\node (B) at (3, 0) {(231,213)};
\node (C) at (3,3) {(312,213)};
\node (D) at (0,3) {(132,231)};


\draw [<->] (A) edge (B) (B) edge (C) (C) edge (D) (A) edge (D) ;
\draw [->](-0.2,-0.3) .. controls (0,-1)  .. (0.2,-0.3);
\draw [->](2.8,-0.3) .. controls (3,-1)  .. (3.2,-0.3);
\draw [->](-0.2,3.3) .. controls (0,4)  .. (0.2,3.3);
\draw [->](2.8,3.3) .. controls (3,4)  .. (3.2,3.3);
\put(40,2){{\footnotesize $r$}}
\put(40,90){{\footnotesize $c$}}
\put(90,50){{\footnotesize $i$}}
\put(-10,50){{\footnotesize $i$}}
\put(5,110){{\footnotesize $r$}}
\put(90,-20){{\footnotesize $c$}}
\put(5,-20){{\footnotesize $c$}}
\put(90,110){{\footnotesize $r$}}
\end{tikzpicture}
\caption{}
          \end{subfigure}
          \begin{subfigure}[b]{0.30\textwidth}
\centering
        \begin{tikzpicture}[node distance=4cm]

\node (A) at (0, 0) {(312,321)};
\node (B) at (3, 0) {(213,123)};
\node (C) at (3,3) {(231,321)};
\node (D) at (0,3) {(132,123)};


\draw [<->] (A) edge (B) (B) edge (C) (C) edge (D) (A) edge (D) (C) edge (A);

\draw [->](-0.2,3.3) .. controls (0,4)  .. (0.2,3.3);
\draw [->](2.8,-0.3) .. controls (3,-1)  .. (3.2,-0.3);
\put(40,2){{\footnotesize $r$}}
\put(40,90){{\footnotesize $r$}}
\put(40,50){{\footnotesize $i$}}
\put(90,50){{\footnotesize $c$}}
\put(-10,50){{\footnotesize $c$}}
\put(0,110){{\footnotesize $i$}}
\put(90,-20){{\footnotesize $i$}}
\end{tikzpicture}
\caption{}
       \end{subfigure}
       
\bigskip

\begin{center}          
          \begin{subfigure}[b]{0.30\textwidth}
\centering
\begin{tikzpicture}[node distance=4cm]

\node (A) at (0,2.5) {(132,213)};
\node (B) at (3,2.5) {(231,312)};


\draw [<->]  (B) edge (A) ;
\draw [->](-0.2,2.8) .. controls (0,3.5)  .. (0.2,2.8);
\draw [->](2.8,2.8) .. controls (3,3.5)  .. (3.2,2.8);
\put(88,95){{\footnotesize $i$}}
\put(4,95){{\footnotesize $i$}}
\put(40,75){{\footnotesize $r, c$}}
\end{tikzpicture}
\caption{}
          \end{subfigure}
          \begin{subfigure}[b]{0.30\textwidth}
\centering
\begin{tikzpicture}[node distance=4cm]

\node (A) at (3,2.5) {(123,321)};

\draw [->](2.8,2.8) .. controls (3,3.5)  .. (3.2,2.8);
\put(90,95){{\footnotesize $r, c, i$}}
\end{tikzpicture}
\caption{}
          \end{subfigure}
\end{center}          
\caption{The bijections among the subclasses $S_n(\taua,\taub)$ for $\taua,\taub\in S_3$ 
under the operations $i$ = inverse, $c$ = complement, $r$ = reverse.} 
\label{graph}
\end{figure}
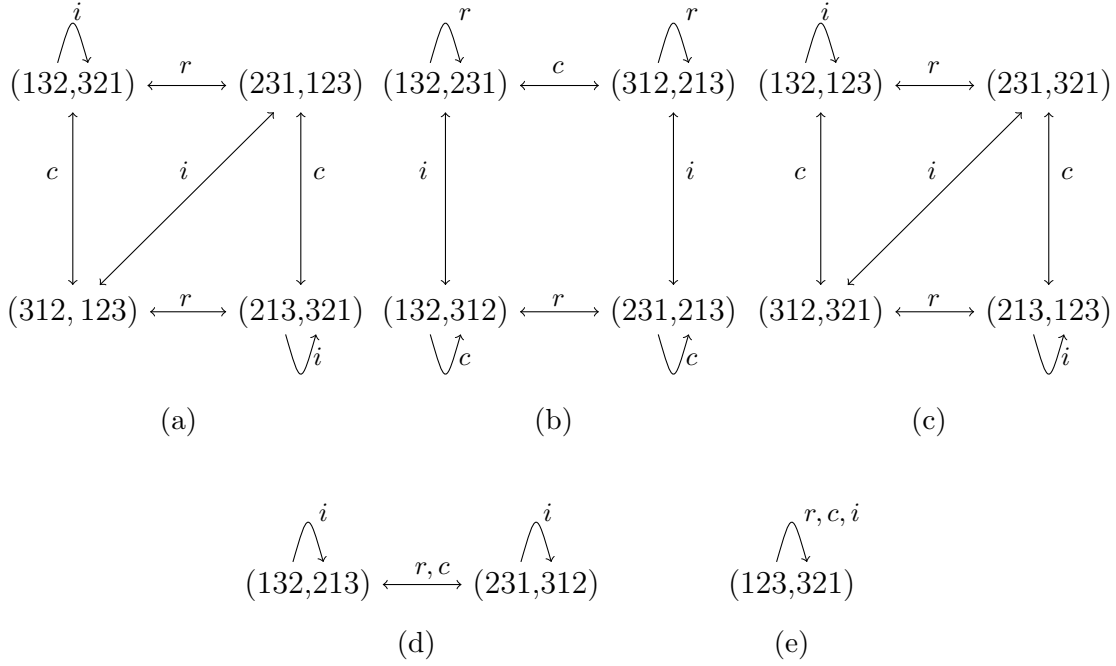

Our results are less precise for $\LIS_n$ and $\LDS_n$ on $S_n(\tau)$ with longer patterns $\tau$.
For $\tau\in S_k$ with $\tau_1=k$, we prove (see Corollary \ref{liscor}) that 
\begin{equation}
    \label{eq.lilis}
      \liminf_{n\rightarrow\infty}  \frac{\E_n^{\tau}(\LIS_n)}{n}   \;\geq \;  \frac{1}{L(\tau)} \,.
\end{equation}
In the special case $\tau=k(k-1)\cdots 21$, the Schensted correspondence
leads to exact asymptotics, and in particular to the result that $\LIS_n/n$ converges in
probability to $1/(k-1)$ (see Theorem \ref{monotonpatt}).
We conjecture that the behaviour of Equation (\ref{eq.lilis}) is generic, in the sense that 
at least one of $\E_n^{\tau}(\LIS_n)/n$ or $\E^{\tau}_n(\LDS_n)/n$ is bounded away from
0 for any pattern $\tau$.

\subsection{Longest alternating subsequences}
   \label{sec-intalt}
In 2006, an analogous theory for alternating subsequences in $S_n$ with the uniform probability measure was developed by Stanley~\cite{RS} and Widom~\cite{W}. For a given $\sigma \in S_n,$ we say that $\sigma_{i_1}\sigma_{i_2}\cdots\sigma_{i_k}$ is an \textit{alternating subsequence of length} $k$ in $\sigma$ if $i_1<i_2<\cdots<i_k$ and $\sigma_{i_1}>\sigma_{i_2}<\sigma_{i_3}>\sigma_{i_4}\cdots \sigma_{i_k}$. 
Let $\LAS_n(\sigma)$ be the length of the longest alternating subsequence in $\sigma$. Stanley proved that $\E(\LAS_n)=\frac{4n+1}{6}$ for $n\geq2$ and $\SD(\LAS_n)=\sqrt{\frac{8}{45}n-\frac{13}{180}}$ for $n\geq 4$. Furthermore, Widom proved that
$\LAS_n$ is asymptotically normal.

In \cite{FMW}, Firro, Mansur and Wilson studied the longest alternating subsequence problem for
 pattern-avoding permutations. They showed that for each $\tau\in S_3,$  
 $\LAS_n$ on $S_n(\tau)$ is asymptotically normal with mean $\E^{\tau}_n(\LAS_n)\sim n/2$ and 
 standard deviation $\SD^{\tau}_n(\LAS_n)\sim \sqrt n/2.$
For the exact values of the means and standard deviations, see \cite{FMW}.
We are not aware of any existing work on this problem for other classes of pattern-avoiding permutations.

We consider $\LAS_n$ on $S_n(T)$ for $T \subset S_3$ with $|T|=2$.  Some cases
such as $T=\{132,231\}$ force $\LAS_n\leq 3$ for every $n$.  Barring such cases,
we show that $\LAS_n$ is asymptotically normal with mean $n/2$ and standard deviation
$\sqrt{n}/2$ (Theorem \ref{lasthm}).  

For patterns $\tau$ of length 4 or more, our results for $\LAS_n$ on $S_n(\tau)$ are less
complete.  In some cases, including $4231$ and $2413$, we can show that 
the expected value of $\LAS_n$ is at least $cn$ for some positive constant $c$ (see
Theorem \ref{thm.alt4}).  
We conjecture that this is true for every pattern (except 12 and 21).

\subsection{Rare regions}
    \label{sec-intrare}
A permutation $\sigma \in S_n$ can be visualized via its plot,
which is the set $\{(i,\sigma_i):i\in [n]\}$ of $n$ points in the plane.
 Plots of randomly generated $\tau$-avoiding permutations
are shown in Figure~\ref{regionfig} for some patterns $\tau$ of length 4.  
Such plots suggest that for some patterns, there can be large regions 
of $[n]^2$ that rarely contain any points of randomly generated members of $S_n(\tau)$. Let's first introduce some definitions.

\begin{figure}[!h]
\begin{subfigure}{0.55\textwidth}
\includegraphics[width=6.4cm, height=6.4cm]{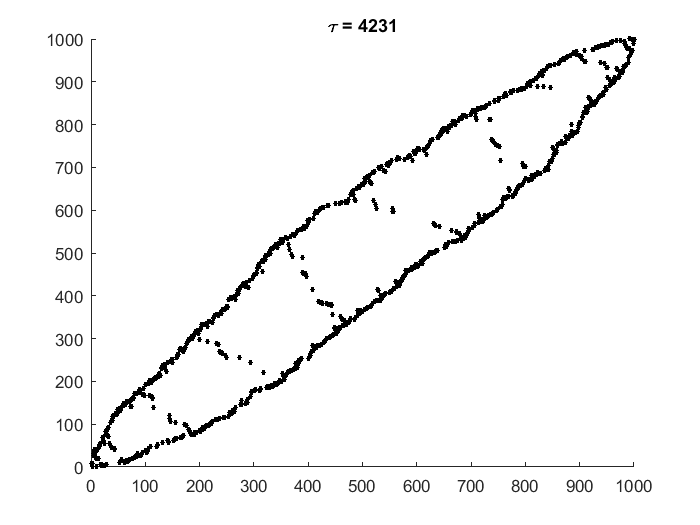}
\end{subfigure}
\begin{subfigure}{0.55\textwidth}
\includegraphics[width=6.4cm, height=6.4cm]{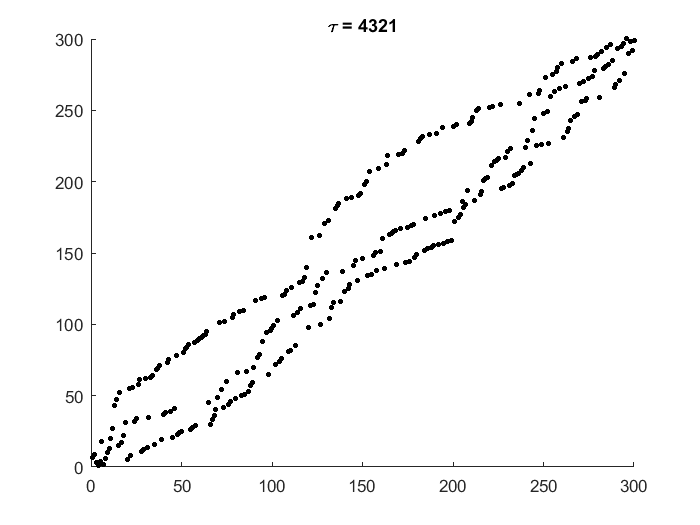} 
\end{subfigure}
\begin{subfigure}{0.55\textwidth}
\includegraphics[width=6.4cm, height=6.4cm]{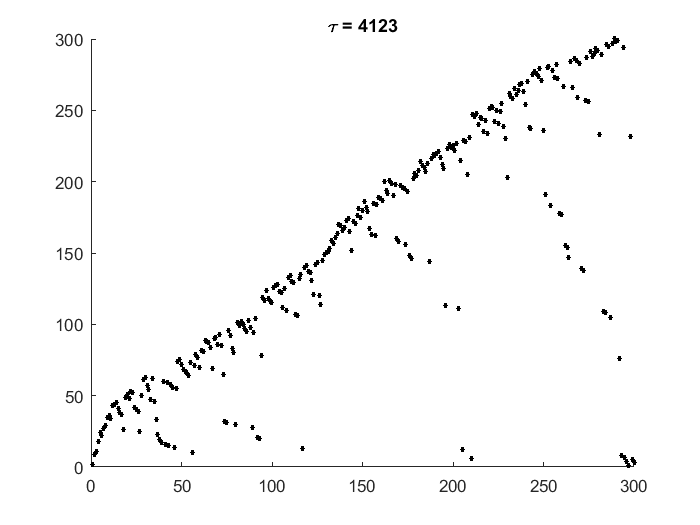}
\end{subfigure}
\begin{subfigure}{0.55\textwidth}
\includegraphics[width=6.4cm, height=6.4cm]{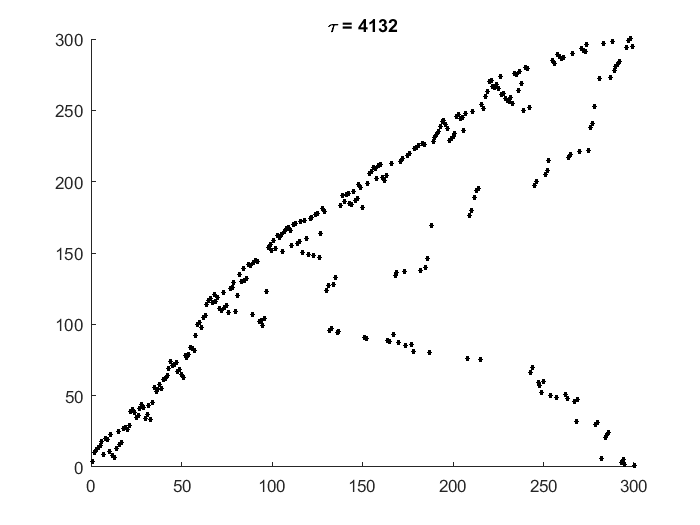} 
\end{subfigure}
\begin{subfigure}{0.55\textwidth}
\includegraphics[width=6.4cm, height=6.4cm]{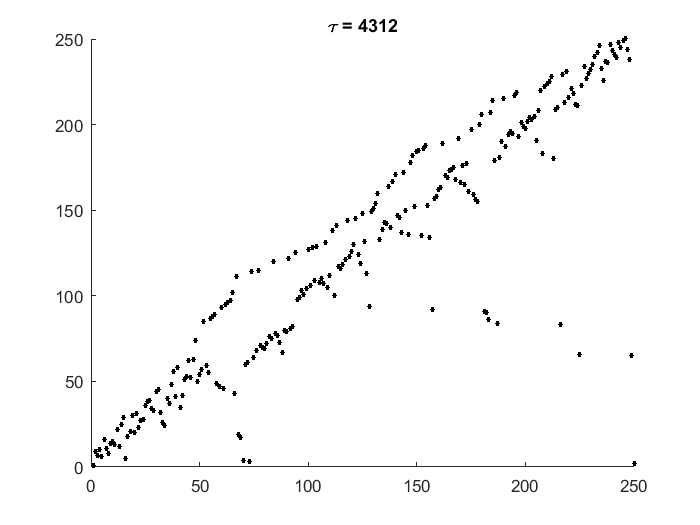}
\end{subfigure}
\begin{subfigure}{0.55\textwidth}
\includegraphics[width=6.4cm, height=6.4cm]{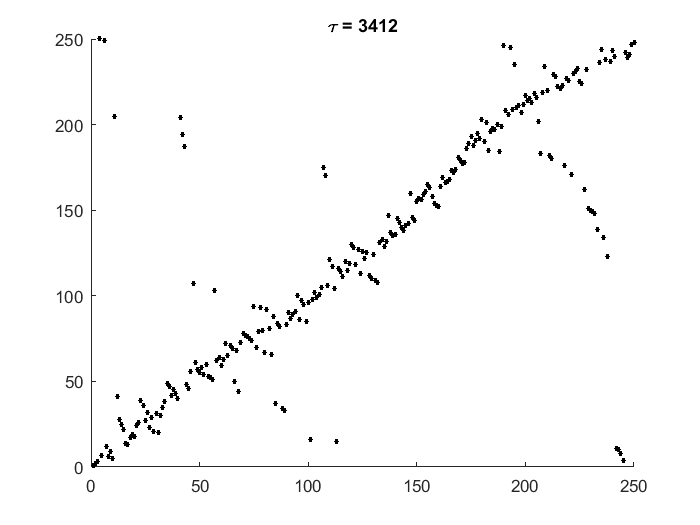} 
\end{subfigure}
\caption{Examples of randomly generated permutations in $S_{n}(\tau)$ for 
some patterns $\tau\in S_4$. The top left figure was generated by Yosef Bisk
under the supervision of N. Madras, using a modification
of the Monte Carlo algorithm of \cite{ML}.}
\label{regionfig}
\end{figure}

To deal with these concepts, we scale the plot of $\sigma$ down to the unit square,
and consider what parts of the unit square are likely to remain empty.
We shall say that a point $(x,y)$ of $[0,1]^2$ is  $\tau$-rare if 
\begin{align*}
\text{ for every sequence } & \{(I_n,J_n)\}_{n\geq1}\text{ such that }(I_n,J_n)\in[n]^2\text{ and }   
\lim_{n\to\infty}\left(\frac{I_n}{n},\frac{J_n}{n}\right)=(x,y), \\
&\text{ we have }
\limsup_{n\to\infty}\bP_n^{\tau}(\sigma_{I_n}=J_n)^{1/n} \,<\, 1 \, ;
\end{align*}
that is, 
every sequence $(I_n,J_n)$ of grid points that scales down to $(x,y)$ has
exponentially decaying probabilities of being in the plot of $S_n(\tau)$.  
Then the ``rare region'' $\mathcal{R}\equiv\mathcal{R}(\tau)$ is defined to be the
set of all $\tau$-rare points in the unit square.   Let 
$\mathcal{R}^{\uparrow}=\mathcal{R}\cap \{(x,y)\in [0,1]^2:y>x\}$ be the part of the 
rare region above the diagonal $y=x$.
Using this terminology, we have the following basic result from \cite{AtM}
(see Theorems 1.3 and 8.1 and Proposition 3.1).

\begin{theorem}[{\cite{AtM}}]
    \label{thm.atm}
Assume $\tau\in S_k$ and $\tau_1>\tau_k$.
\\
(a)  Assume $\tau_1=k$.  Then there is a $\delta>0$ such that 
$[0,\delta]\times[1-\delta,1]\subset \mathcal{R}$; that is, all points
sufficiently close to the point $(0,1)$ are $\tau$-rare.
\\
(b) Assume $\tau_1<k$.  Then $\mathcal{R}^{\uparrow}=\emptyset$; that is,
there are no $\tau$-rare points above the diagonal.
\end{theorem}

Without loss of generality, we shall assume for the rest of this paragraph
that $\tau \in S_k$ and $\tau_1>\tau_k$. We define the ``good region'' $\mathcal{G}$ to be the complement of  the rare region:
$\mathcal{G}=[0,1]^2\setminus \mathcal{R}$.   We prove  that the 
region $\mathcal{G}$ contains the diagonal $y=x$ and is a closed set (Theorems \ref{regionthm}
and \ref{thm.boundary}).  When $\tau_1=k$,  we prove that the boundary between 
$\mathcal{R}^{\uparrow}$ and $\mathcal{G}$ is a curve $y=r^{\uparrow}(x)$ that is Lipschitz 
continuous and strictly increasing on $[0,1]$.  Moreover, the left and right derivatives of 
$r^{\uparrow}$ are in the 
interval $[L(\tau)^{-1},L(\tau)]$ at every point (see Theorem~\ref{regionthm}).

\subsection{Notation and terminology}
\label{sec-not}
This section contains basic notation and terminology used in this paper, much of
which is standard in the permutation pattern literature.

For $\tau\in S_k$ and $\sigma\in S_n$, we say that $\sigma$ \textit{contains the pattern} $\tau$ if there is a subsequence $\sigma_{i_1}\sigma_{i_2}\cdots\sigma_{i_k}$ of $k$ elements of $\sigma$ that appears in the same relative order as the pattern $\tau$. For example, the permutation $\sigma=6431257$ contains the patterns $321$ and $3124$ (since $\sigma$ contains the subsequences $\sigma_1\sigma_2\sigma_3=643$ or $\sigma_1\sigma_3\sigma_5=632,$ and $\sigma_2\sigma_4\sigma_5\sigma_7=4127$). We say that $\sigma$ \textit{avoids the pattern} $\tau$ if it does not contain $\tau$. The set of all permutations in $S_n$ avoiding $\tau$ is denoted by $S_n(\tau).$ For any set $T$ of patterns, we write $S_n(T)$ to denote the set of permutations in $S_n$ avoiding all the patterns in $T$, that is, $S_n(T)=\cap_{\tau \in T}S_n(\tau)$. For example, the permutation $521346$ is in $S_6(132,2314)$ because it avoids both $132$ and $2314$. 

The \textit{direct sum} of two permutations $\sigma\in S_n$ and $\phi\in S_m$ is the 
permutation $\sigma\oplus \phi$ in $S_{n+m}$ obtained by concatening $\phi$ to the 
northeast corner of $\sigma$,
\[    \sigma\oplus \phi   \;=\;   \sigma_1\cdots\sigma_n(n+\phi_1)\cdots(n+\phi_m) \,,  \]
while the \textit{skew sum} concatenates $\phi$ to the southeast corner of $\sigma$,
\[    \sigma\ominus \phi   \;=\;   (m+\sigma_1)\cdots(m+\sigma_n) \phi_1\cdots\phi_m \,.  \]
A permutation is \textit{layered} if it is the direct sum of one or more decreasing 
permutations.

For a given permutation $\sigma$, we say that $\sigma_i$ is  a \textit{left-to-right maximum} if 
$\sigma_i>\sigma_j$ for all $j<i.$ In this situation, $i$ is referred as the \textit{location} of the 
left-to-right maximum.  (For clarity, we sometimes refer to $\sigma_i$ as the \textit{height}.)
Similarly, we say that $\sigma_i$ is a \textit{right-to-left maximum} if 
$\sigma_i>\sigma_j$ for all $j>i.$ 

For natural numbers $n$, $i$, and $j$, we define $S_n(\tau; \sigma_i=j)$ to be the 
set of permutations $\sigma$ in $S_n(\tau)$ such that $\sigma_i=j$.  
More generally, for any statement $\mathcal{P}$ about a generic permutation $\sigma$,
we let $S_n(\tau; \mathcal{P})$ be the set of permutations in $S_n(\tau)$ for 
which $\mathcal{P}$ is true.    

The number of elements in a set $A$ is denoted by $|A|.$ 

For two sequences $\{a_n\}_{n\geq1}$ and $\{b_n\}_{n\geq1}$, we write $a_n\sim b_n$ if 
$\lim_{n\to\infty}\frac{a_n}{b_n}=1.$

\medskip

\section{Results and proofs}
In this section, we will state and prove our        results.

\subsection{Longest monotone subsequences of permutations in $S_n(\taua,\taub)$ for \\
$\taua, \taub\in S_3$}
    \label{sec.tau1tau2}

In Table~\ref{summary}, we summarized our results on the mean and 
standard deviation of $\LIS_n$ and 
$\LDS_n$ on $S_n(\taua,\taub)$ for $\taua,\taub\in S_3$, and whether they are asymptotically 
normal or not.

Theorem~\ref{MainTh1} below gives fuller descriptions of the distributions of the longest 
monotone subsequences in these cases.  To prepare for the statement of  the theorem, we first
introduce the relevant distributions.

We use $\Bin(n,p)$ to denote a binomial random variable with parameters $n$ and $p$.

For $n\geq 3,$ let $\mathcal{D}[n]$ denote the triangular set of lattice points above the diagonal 
in the square $[n]\times [n]$, together with the origin; 
that is, $\mathcal{D}[n]=\{(k,m)\in \mathbb{Z}^2: 1\leq k<m\leq n\}\cup\{(0,0)\}.$ 
Then $|\mathcal{D}[n]|={n \choose 2}+1.$ 
We put uniform probability measure on $\mathcal{D}[n]$ and consider the following random variable:
\[\D_n((k,m))\;=\; n-\min(m-k,k) \;=\; \max(n-m+k,n-k).
\]
Then we have
\[\bP(\D_n=n-j)\;=\;\frac{2n-4j+1}{{n \choose 2}+1} \hspace{3mm}\text{ for } 1\leq j\leq \frac{n}{2}\,,
\hspace{5mm}\hbox{and}\hspace{5mm}
\bP(\D_n=n)\;=\;\frac{1}{{n \choose 2}+1}\,.
\]
Here is a way to think about $\D_n$, given that $(k,m)\neq (0,0)$.  Choose two points uniformly
without replacement from $[n]$; let $k$ be the smaller number and let $m$ be the larger.
Then $\D_n$ is the number of points remaining in $[n]$ after discarding the smaller of the 
intervals $(0,k]$ and $(k,m]$.

Let $X_1,X_2,\cdots$ be an independent sequence of Bernoulli$(\frac{1}{2})$ random variables:
\begin{equation}
   \label{eq.bern}
    \bP(X_i=1)\;=\;\bP(X_i=0)\;=\;\frac{1}{2} \hspace{5mm}\text{ for all }i\geq 1.
\end{equation}

Then define the following random variables for $n\geq 3$:
\begin{align}
\T_n&=n-1-\sum_{i=2}^{n-2}X_i+\sum_{i=2}^{n-1}X_iX_{i-1}\\
\R_{n}&=\text{the length of the longest run of zeros in }X_1,X_2,\cdots,X_{n}.
\end{align}
 
 For two random variables $X$ and $Y$, we write $X\myeq Y$ to say that $X$ and
 $Y$ have the same distribution.

\begin{theorem} 
\label{MainTh1}
Let T denote a 2-element subset of $S_3.$ Consider $S_n(T)  $ with the uniform probability measure. Then we have the following:
\begin{itemize}
\item[(a)] For $T=\{132, 321\},$ $\LIS_n\myeq \D_n$ and $\LDS_n\leq2$.
\item[(b)] For $T=\{132, 231\},$ $\LIS_n\myeq \LDS_n\myeq \Bin(n-1,1/2)+1.$
\item[(c)] For $T=\{132,123\},$ $\LDS_n\myeq \T_{n}$ and $\LIS_n\leq 2$.
\item[(d)] For $T=\{132, 213\},$ $\LDS_n\myeq \Bin(n-1,1/2)+1$ and $\LIS_n\myeq \R_{n-1}+1$.
\end{itemize}
\end{theorem}

\begin{figure}[h!]
\begin{subfigure}{0.51\textwidth}
\includegraphics[width=8.5cm, height=8.6cm]{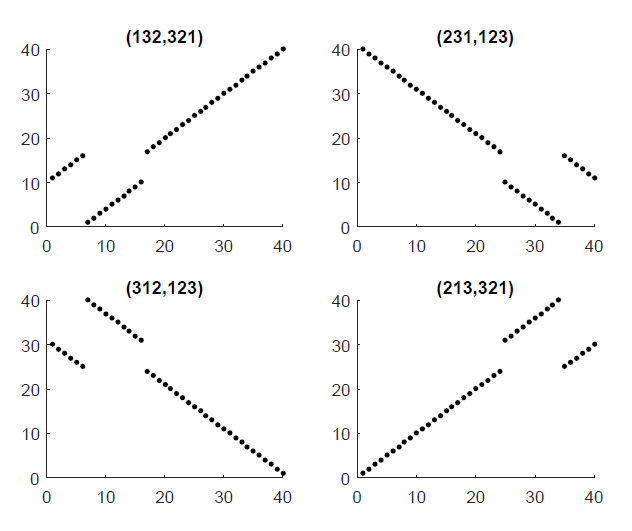}
\caption{}
\label{a}
\end{subfigure}
\begin{subfigure}{0.50\textwidth}
\includegraphics[width=8.5cm, height=8.6cm]{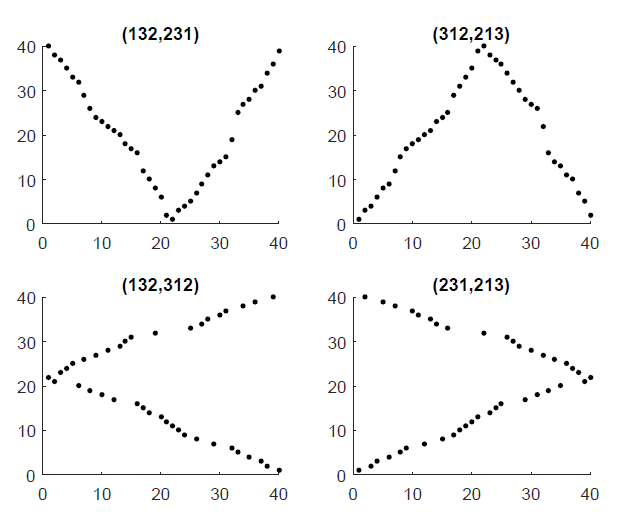} 
\caption{}
\label{b}
\end{subfigure}
\begin{subfigure}{0.50\textwidth}
\includegraphics[width=8.5cm, height=8.6cm]{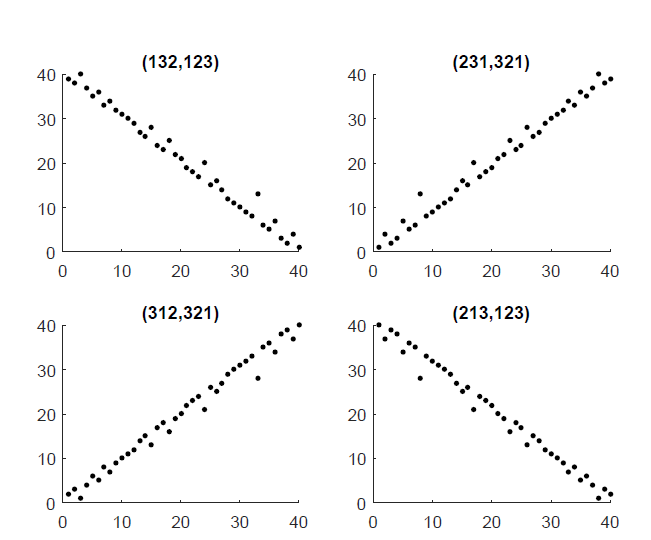} 
\caption{}
\label{c}
\end{subfigure}
\begin{subfigure}{0.50\textwidth}
\includegraphics[width=8.7cm, height=8.8cm]{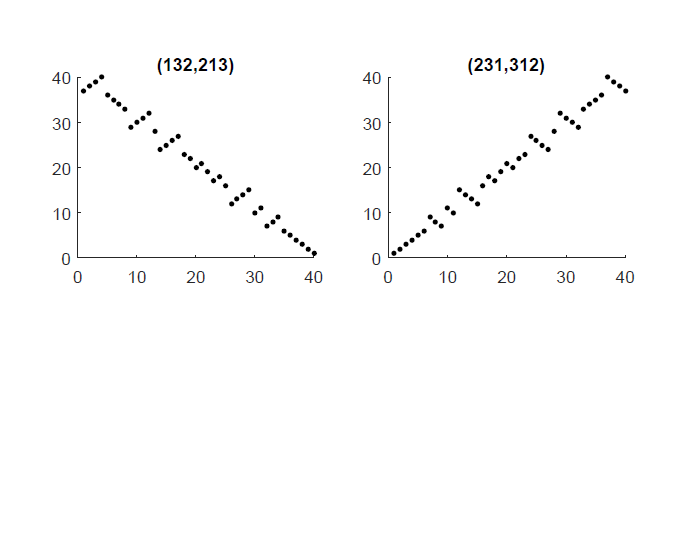}
\caption{}
\label{d} 
\end{subfigure}
\caption{Examples of permutations from $S_n(\taua,\taub)$ for $\taua,\taub\in S_3$ and $n=40.$ The permutations from $S_{40}(132,321), S_{40}(132,231), S_{40}(132,123)$ and $S_{40}(132,213)$ were generated randomly, and the rest were obtained using bijections of Figure~\ref{graph}.}
\label{lastfigure}
\end{figure}

The proof of Theorem \ref{MainTh1} exploits the nice structure of each  $S_n(T)$.  
First, we note that Simion and Schmidt \cite{SS} showed 
 \[|S_n(132,321)|={n \choose 2}+1\quad \text{and} \quad |S_n(132,231)|=|S_n(132,123)|=|S_n(132,213)|=2^{n-1}.\]
The proofs of each part of Theorem \ref{MainTh1} includes an explicit bijection for $S_n(\tau)$: 
with $\mathcal{D}[n]$ for part (a), and with
binary strings of length $n-1$ for the other parts.  We shall use these bijections again in the proof of 
Theorem \ref{lasthm} in Section \ref{LASsection} about longest alternating subsequences.

\begin{proof}[Proof of Theorem \ref{MainTh1}]
\textit{(a) The case $T=\{132,321\}$:} 
We will follow Propositions 11 and 13 in \cite{SS} which shows that there exists a bijection between $S_n(132,321)\setminus\{12\cdots n\}$ and the 2-element subsets of $[n].$ 
Let $\sigma \in S_n(132,321).$ 

If $\sigma$ is not the identity permutation, let $m$ be the largest
$i$ such that $\sigma_i\neq i$, and let $k$ be the index such that $\sigma_k=m$.
Then we must have $\sigma_1\sigma_2\cdots\sigma_{k-1}=(m-k+1)(m-k+2)\cdots(m-1)$ and 
$\sigma_{k+1}\sigma_{k+2}\cdots\sigma_{m}=12\cdots(m-k)$. 
That is, $\sigma$ must have the form
\[\sigma=(m-k+1)(m-k+2)\cdots m 1 2 \cdots (m-k) (m+1)(m+2)\cdots n\]
for some $1\leq k\leq m\leq n,$ and $m=k$ iff $\sigma=12\cdots n.$ 
See  Figure~\ref{fig132231}(a) for an example.

For $1\leq k< m\leq n,$ the pair $\{k,m\}$ determines $\sigma$ uniquely.
It follows from this argument that with uniform probability measure on $S_n(132,321)$, $\LIS_n\myeq \D_n.$

\begin{remark}
We note that $\sigma\in S_n(132,321)$ can be represented as the \textit{inflation} of $213$ by 
the increasing permutations $\taua, \taub, \tauc$ where $\taua=12\cdots k$, 
$\taub=12\cdots (m-k)$, $\tauc=12\cdots (n-m)$; that is, $\sigma=213[\taua,\taub,\tauc]$. 
The permutations in $S_n(132,321)$ are also examples of $3\times3$ \textit{monotone grid classes}
 (see section 4 of \cite{Vatter}).
\end{remark}

\begin{figure}[t!]
\begin{subfigure}{0.55\textwidth}
\includegraphics[width=7cm, height=7cm]{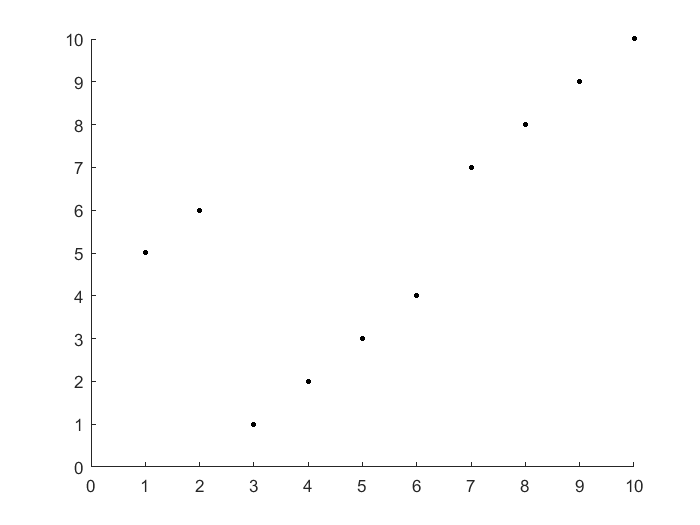} 
\label{a}
\caption{}
\end{subfigure}
\begin{subfigure}{0.55\textwidth}
\includegraphics[width=7cm, height=7cm]{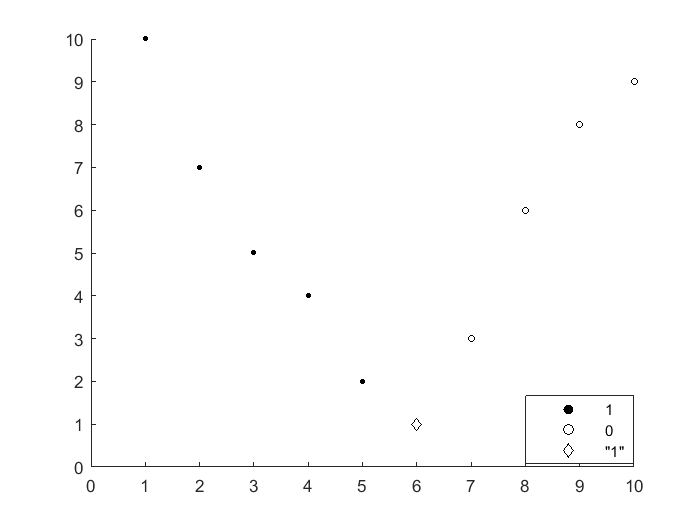} 
\label{b}
\caption{}
\end{subfigure}
\caption{The left figure is the plot of the permutation
$\sigma= 5\,     6\,  1\,    2\,    3\, 4\,  7\,    8\,    9\,    10\,$ in $S_{10}(132,321)$ where $k=2$ and $m=6$. 
The right figure is the plot of the permutation $\sigma=10\,    7\,    5\,    4\,    2\,    1\,     3\,     6\, 8\,     9$ in $S_{10}(132,231)$ which corresponds to the $0$-$1$ sequence of length $9,$ $101101001$.}
\label{fig132231}
\end{figure}

\medskip
\noindent
\textit{(b) The case $T=\{132,231\}$:}
Let $\sigma\in S_n(132,231).$ Let $k$ be the index such that  $\sigma_k=1$. 
To avoid the pattern $132$, we must have $\sigma_{k+1}<\sigma_{k+2}<\cdots<\sigma_n;$ 
and to avoid the pattern $231,$ we must have $\sigma_1>\sigma_2>\cdots>\sigma_{k-1}.$ 
In the terminology of 
\cite{Bevan}, $S_n(132,231)$ is a \textit{skinny monotone grid class}, that is, it is the \textit{juxtaposition} of a decreasing and an increasing sequence (see chapter 3 of \cite{Bevan}).

For our permutation  $\sigma$ 
we obtain a $0$-$1$ sequence of length $n-1$ by labeling the decreasing points in the plot of 
$\sigma$ with $1$ and the increasing points with $0$, and reading them from bottom to top, 
excluding the point $\sigma_k=1$.
That is, for $j\in [n{-}1]$, let $i$ be the index such that $\sigma_i=j+1$, and 
then set $x_j$ to be 1 (respectively, 0) if  $i<k$ (respectively, $i>k$).  
See Figure \ref{fig132231}(b) for an example.
It is not hard to see
that this gives a bijection between $S_n(132,231)$ and $\{0,1\}^{n-1}$.  Thus the uniform distribution
ensures that every sequence has probability $2^{-(n-1)}$ and hence that 
$x_1,\ldots,x_{n-1}$ are independent Bernoulli$(\frac{1}{2})$ random variables.

Then we have 
\[\LDS_n(\sigma)\;=\; k+1 \;=\; \sum_{i=1}^{n-1}1_{x_i=1}+1
\hspace{4mm}\text{ and } \hspace{4mm}
\LIS_n(\sigma)\;=\; l+1 \;=\;\sum_{i=1}^{n-1}1_{x_i=0}+1\,,\]
and hence  $\LDS_n\myeq \LIS_n\myeq \Bin(n-1,1/2)+1.$

\begin{figure}[t!]
\begin{subfigure}{0.55\textwidth}
\includegraphics[width=7cm, height=7cm]{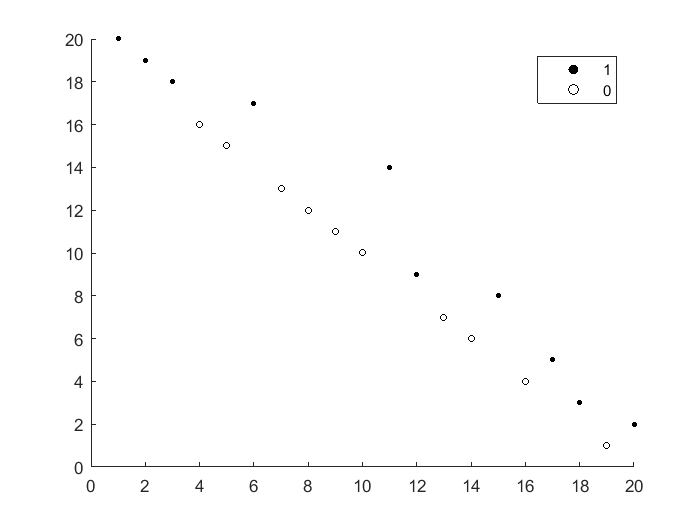}

\caption*{(c)}
\end{subfigure}
\begin{subfigure}{0.55\textwidth}
\includegraphics[width=7cm,height=7cm]{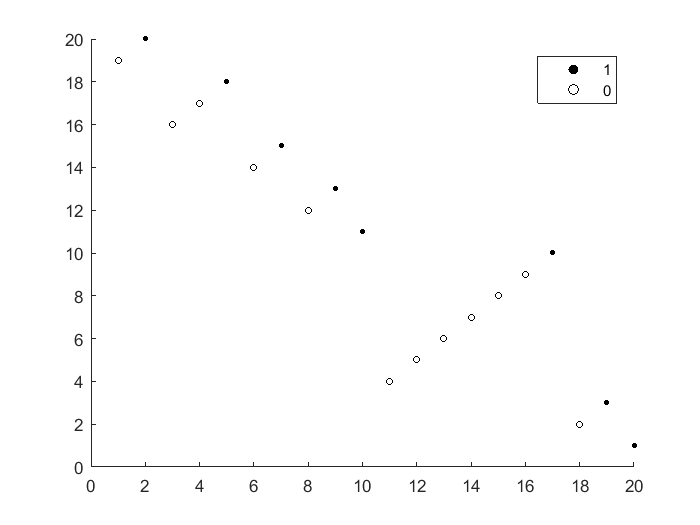}

\caption*{(d)} 
\end{subfigure}
\caption{The left figure is the plot of the permutation $\sigma=20\,19\,    18\,    16\,     15\,     17\,    13\,     12\,     11\,     10\,     14\,      9\,      7\,      6\,      8\,      4\,      5\,      3\,      1\,     2\,$ in $S_{20}(132,123)$ which corresponds to the $0$-$1$ sequence of length $19,$ $1     1     1     0     0     1     0     0     0     0     1     1     0     0     1     0     1     1     0  .$
The right figure is the plot of the permutation $\sigma=19\,    20\,    16\,    17\,    18\,    14\,    15\,    12\,    13\,    11\,     4\,    5\,     6\,     7\,     8\,     9\,    10\,     2\,     3\,     1\,$ in $S_{20}(132,213)$ which corresponds to the $0$-$1$ sequence of length $19,$ $0     1     0     0     1     0     1     0     1     1     0     0     0     0     0     0     1     0     1.$ }
\label{fig123132}
\end{figure}

\medskip
\noindent
\textit{(c)  The case $T=\{132, 123\}$:}  
Let $\sigma\in S_n(132,123).$ Suppose $\sigma_k=n$ for some $k\in [n]$. 
To avoid the pattern $123$, we must have $\sigma_1>\sigma_2>\cdots>\sigma_{k-1};$ and 
to avoid $132,$ we must have $\sigma_i>n-k$ whenever $i<k.$ Therefore, recalling notation from
Section \ref{sec-not}, 
we must have $\sigma=(\delta\oplus 1)\ominus \phi$ where $\delta$ is a decreasing permutation
(possibly of length 0), 1 is the one-element permutation, and  $\phi$ avoids $123$ and $132$. 
Iterating this argument, we see that $\sigma$ is the skew sum of one or more permutations
of the form $\delta\oplus 1$.
In the terminology of \cite{AJ}, 
$\cup _n S_n(132,123)$ is the class $\ominus(\mathbf{D}\oplus \mathbf{1})$.
Observe that each ``$\oplus \mathbf{1}$'' in this decomposition corresponds to a right-to-left 
maximum of $\sigma$.

We  construct a map from  $S_n(132,123)$ to the set of $0$-$1$ sequences  
$x_1x_2\cdots x_{n-1}$ as follows.   Given $\sigma\in S_n(132,123)$, let 
$1\leq i_1<i_2<\cdots<i_k\leq n-1$ be the locations in $[n-1]$ of its right-to-left maxima.  
Then define $x_{i_1}=x_{i_2}=\cdots =x_{i_k}=1$ and $x_j=0$ for $j \notin \{ i_1,i_2,\cdots,i_k\}.$
It is not hard to see that this map is a bijection.
Moreover, fixing $x_n=x_0=1,$  we have 
\begin{align*}
\LDS_n(\sigma)&=\sum_{j=1}^{k+1}\max(1,i_j-i_{j-1}-1)\\
&= \sum_{j=1}^{k+1}(i_j-i_{j-1}-1)+\sum_{j=1}^{k+1}1_{i_j=i_{j-1}+1}\\
&= \sum_{i=1}^{n-1}1_{x_i=0}+\sum_{i=1}^{n}1_{x_{i-1}=x_i=1}\\
&= 2+\sum_{i=2}^{n-2}1_{x_i=0}+\sum_{i=2}^{n-1}1_{x_{i-1}=x_i=1}\\
	&=2+\sum_{i=2}^{n-2}(1-x_i)+\sum_{i=2}^{n-1}x_ix_{i-1}\\
	&=n-1-\sum_{i=2}^{n-2}x_i+\sum_{i=2}^{n-1}x_ix_{i-1}.
\end{align*}
Therefore with the uniform probability measure on $S_n(132,123)$, $\LDS_n\myeq \T_{n}.$

\begin{example} For the $0$-$1$ sequence $0 1 1 0 1$ of length $5$,  
the corresponding permutation of length $6$ in $S_{6}(132,123)$ is 
$\sigma=5\,      6\,      4\,      2\,      3\,      1$. See also the example in Figure~\ref{fig123132}(c).
\end{example}

\medskip
\noindent
\textit{(d) The case $T=\{132,213\}$:}
Let $\sigma\in S_n(132,213).$ Suppose $\sigma_k=n$ for some $k\in [n]$. To avoid the pattern 
$213$, we must have $\sigma_{1}<\sigma_{2}<\cdots<\sigma_k;$ and  to avoid $132,$ 
we must have $\sigma_i>n-k$ whenever $i<k.$ Therefore, $\sigma=\lambda\ominus \phi$
where $\lambda$ is the increasing permutation of length $k$, and $\phi\in S_{n-k}(132,213)$.
Iterating this argument, we see that $\sigma$ is the skew sum of one or more increasing
permutations.  Thus, this class is the reverse (or complement) of the layered permutations.

Observe also that
a member of $S_n(132,213)$ is uniquely determined by the locations of its right-to-left maxima. Therefore, as in the proof of part (c), we obtain a 
a bijection between $S_n(132,213)$ and the set of $0$-$1$ sequences $x_1x_2\cdots x_{n-1}$
by setting $x_i$ to be 1 (respectively 0) if $\sigma_i$ is (respectively, is not) 
a left-to-right maximum of $\sigma$.
Then we have 
\[\LDS_n(\sigma)=\sum_{i=1}^{n-1}1_{x_i=1}+1\]
and 
\[\LIS_n(\sigma)=1+\text{length of the longest run of zeros in } x_1x_2\cdots x_{n-1}.\]
Therefore,  with the uniform probability measure on $S_n(132,213)$, we have 
$\LDS_n\myeq \Bin(n-1,1/2)+1$ and $\LIS_n\myeq \R_{n-1}+1.$ 
\end{proof}

\begin{example}The permutation $\sigma=5\,     6\,  1\,     2\,     3\, 4$ in $S_{6}(132,213)$
 corresponds to the $0$-$1$ sequence $0 1 0 0 0$ of length $5$.  
See also the example in Figure \ref{fig123132}(d).
\end{example}

Recall that  a \textit{composition} of an integer $n$ is a way of writing $n$ as the sum of an ordered
 sequence of positive integers. Two sequences with the same terms but in different order correspond to different compositions but to  the same \textit{partition} of their sum. 
 It is known that each positive integer $n$ has $2^{n-1}$ distinct compositions, and that there is an explicit bijection between the compositions of $n$ and the set of layered permutations of length $n$. Probabilistic properties of  random integer compositions and partitions have been studied in literature. 
 From this viewpoint, our results for group (d) correspond to results on page 337 and in 
 Theorem 8.39 of \cite{ToHe}.

We conclude this subsection by explaining how the results of Table \ref{summary} follow from 
Theorem \ref{MainTh1} and some classical results.
 The mean and standard deviation of $\D_n$ are asymptotically
\begin{align*}
\E(\D_n) \sim \frac{5n}{6} \text{ and } \SD(\D_n)\sim \frac{5n}{6 \sqrt 2}.
\end{align*}
It is easy to show that 
\begin{align*}
\E(\T_{n})&=\frac{3n}{4} \text{ and }\SD(\T_{n})\sim\frac{\sqrt{n}}{4}.
\end{align*}
Moreover, the fact that the limiting distribution of  $(\T_n-\E(\T_n))/\SD(\T_n)$ is standard normal 
is a consequence of the following lemma, which can be found for example in  \cite{Ch}.
\begin{lemma}[\cite{Ch}, Theorem 7.3.1]
\label{Chung}
Suppose that $\{Y_n\}$ is a sequence of $m$-dependent, uniformly bounded random variables such that as $n\to +\infty$
\[\frac{\SD(S_n)}{n^{1/3}}\to +\infty\]
where $S_n=Y_1+Y_2+\cdots +Y_n$. Then $(S_n-\E(S_n))/\SD(S_n)$ converges in distribution to the standard normal distribution.
\end{lemma}

\noindent
Finally,
Erd\"{o}s and Renyi \cite{ER} proved that 
\begin{equation}  
\lim_{n\to \infty}\frac{\R_n}{\log_2 n}=1\text{ a.s.}
\end{equation}
It is also known that as $n$ tends to infinity $\E(\R_n)\sim\log_2n$, $\SD(\R_n)$ converges to a positive constant, and $\R_n-\log_2n$ possesses no limit distribution; for more on this topic see the survey paper \cite{S}.

\subsection{Longest alternating subsequences}
\label{LASsection}

In this section, we will present our results on the asymptotic behaviour of the length of the longest alternating subsequence, $\LAS_n$, on $S_n(\taua,\taub)$ with $\taua, \taub \in S_3$ 
under the uniform probability distribution. For a result on $\LAS_n$ for longer patterns, see Theorem~\ref{thm.alt4} in section~\ref{longerpattern}.

\begin{theorem} 
\label{lasthm}
Let T denote a 2-element subset of $S_3.$ Consider $S_n(T)  $ with the uniform probability measure. Then we have the following:
\begin{itemize}
\item[a)] If $T=\{132,321\},$ then we have $\LAS_n(\sigma)\leq 3$ for any $\sigma\in S_n(T).$
\item[b)] If $T=\{132,231\},$ then we have $\LAS_n(\sigma)\leq 3$ for any $\sigma\in S_n(T).$
\item[] If $T=\{132,312\},$ then $\LAS_n$ is asymptotically normal with the mean $\E^T_n(\LAS_n)\sim n/2$  and the standard deviation $\SD^T_n(\LAS_n)\sim \sqrt{n}/2$.
\item[c)] If $T=\{132,123\},$ then $\LAS_n$ is asymptotically normal with the mean $\E^T_n(\LAS_n)\sim n/2$  and the standard deviation $\SD^T_n(\LAS_n)\sim \sqrt{n}/2$.
\item[d)] If $T=\{132,213\},$ then $\LAS_n$ is asymptotically normal with the mean $\E^T_n(\LAS_n)\sim n/2$  and the standard deviation $\SD^T_n(\LAS_n)\sim \sqrt{n}/2$.
\end{itemize}
\end{theorem}
\begin{remark} Parts (a) through (d) correspond to the labelling in Figures \ref{graph} and
\ref{lastfigure}.
Note that since $|\LAS_n(\sigma)-\LAS_n(\sigma^c)|\leq 1$ and $|\LAS_n(\sigma)-\LAS_n(\sigma^r)|\leq 1$, it suffices to consider one representative pair from each symmetry class in Figure~\ref{graph} except the symmetry class in Figure~\ref{graph}-(b). 
\end{remark}

The proof uses the structures exploited in the proof of Theorem \ref{MainTh1}.  The proof of part (b)
needs a bit of new work, so we prove it last.

\begin{proof}[Proof of Theorem \ref{lasthm}] \textit{(a)} Assume $\sigma\in S_n(132,321)$. 
Recall from the proof 
of Theorem~\ref{MainTh1}(a) that $\sigma$ is either the identity or else it can be represented 
as $\sigma=213[\taua,\taub,\tauc]$, the inflation of $213$ by increasing permutations 
$\taua, \taub, \tauc$ (Figure~\ref{lastfigure}(a)). Therefore, $\LAS_n(\sigma)\leq 3$. 

\medskip
\noindent
\textit{(c)}  Recall that the bijection in the proof of Theorem \ref{MainTh1}$(c)$
defines $x_i=1$ if $\sigma_{i}$ is a right-to-left maximum, and $x_i=0$ otherwise, for each $i\in[n-1]$.
Then $x_1, x_2,\cdots, x_{n-1}$ are independent Bernoulli$(\frac{1}{2})$ random variables.

Let $\displaystyle \A_n=\sum_{i=1}^{n-1}1_{x_i=0}1_{x_{i+1}=1}.$ Note that $|\LAS_n-2\A_n|\leq 4.$ 
It is easy to show that $\E(\A_n)\sim\frac{n}{4}$ and $\SD(\A_n)\sim \frac{\sqrt n}{4}.$ Moreover, by Lemma~\ref{Chung},  normalized $\A_n$ converges to the standard normal distribution.
Therefore $\E^T_n(\LAS_n)\sim 2\E(\A_n)$ and 
$\SD^T_n(\LAS_n)\sim2\SD(\A_n),$ and we get the asymptotic normality of $\LAS_n.$

\medskip
\noindent
\textit{(d)} Using the bijection from the proof of Theorem \ref{MainTh1}$(d)$, the proof will be the same as in part $(c)$ above.

\medskip
\noindent
\textit{(b)} Let $\sigma\in S_n(132,231)$.  As shown in the proof of Theorem \ref{MainTh1},  
$\sigma$ is the 
juxtaposition of a decreasing and an increasing subsequence. Therefore $\LAS_n(\sigma)\leq 3$. 

Let $\sigma\in S_n(132,312)$. Note that all the entries greater (respectively, smaller) than $\sigma_1$
 must be in the increasing (respectively, decreasing) order, otherwise we would have a $132$ 
 (respectively, $312$) pattern. 

For $i\in[ n-1]$, define $x_i$ to be $1$ if $\sigma_{i+1}>\sigma_1$, and $0$ otherwise. 
Then $x_1, x_2,\cdots, x_{n-1}$ are  independent Bernoulli$(\frac{1}{2})$ random variables.

Set $\displaystyle \A_n=\sum_{i=1}^{n-1}1_{x_i=1}1_{x_{i+1}=0}.$ Then the rest of the proof 
follows as for part $(c)$.
\end{proof}

\subsection{Results for patterns of length $k\geq4$}
\label{longerpattern}

In this subsection, we will present our results on longer patterns. Theorem \ref{monotonpatt}  concerns monotone patterns.  It yields precise properties 
because we can exploit the Schensted correspondence, which we review below.
The next group of results proves that $\LIS_n$ is unlikely to be $o(n)$ on $S_n(\tau)$ when $\tau$
is of the form $\tau=k\tau_2\cdots \tau_k \in S_k$.  
Finally, Theorem \ref{thm.alt4} proves an
analogous result for $\LAS_n$ on $S_n(\tau)$ for some patterns $\tau$.

First we  recall some basic definitions used in the algebraic combinatorics of permutations. For more on this topic, see chapter 7 of \cite{Bo} or the survey paper \cite{AD}.
A \textit{partition} $\lambda=(\lambda_1,\cdots,\lambda_j)$ of an integer $n\geq 1$ is a sequence of integers with $\lambda_1\geq\lambda_2\geq \cdots \geq\lambda_j\geq1$ and $\sum_i\lambda_i=n.$ A partition $\lambda$ can be identified with its associated \textit{Ferrers shape} of $n$ cells with $\lambda_i$ left justified cells in row $i.$ A (standard) Young Tableau is a Ferrers shape on $n$ boxes in which each number in $[n]$ is assigned to its own box, so that the numbers are increasing within each
row and each column (going down). The Schensted~\cite{Sc} correspondence induces a bijection between permutations and pairs $(P,Q)$ of Young Tableaux of the same shape. The number of Young Tableaux of a given shape $\lambda$, denoted by $d_{\lambda}$, is given by the celebrated \textit{hooklength formula} of Frame and Robinson:
\[d_{\lambda}=\frac{n!}{\prod_ch_c}.\]
The \textit{hook length} $h_c$ of a cell $c$ in a Ferrers shape is the number of cells to the right of $c$ in its row, the cells below $c$ in its column, and the cell $c$ itself. For illustration, in the following Ferrers shape corresponding to the partition $\lambda=(5,4,2)$, each cell $c$ is occupied by its hook length $h_c$. (Lest confusion arise, we note that it is not a Young Tableau.) 
\begin{center}
\ytableausetup{textmode}
\begin{ytableau}
 7&6& 4 &3 &1  \\
 5&4&2 & 1 \\
 2&1
\end{ytableau}
\end{center}

Importantly, if the Schensted correspondence associates the shape $\lambda$ with a given $\sigma \in S_n,$ then $\lambda_1=\LIS_n(\sigma).$ In \cite{Sc}, Schensted shows also that the length of the decreasing subsequences in $\sigma$ are encoded by the $P$-tableau of $\sigma$: For any $\sigma \in S_n$, we have 
\begin{equation}
\label{Schensted} 
P_{\sigma^r}=P^t_{\sigma}
\end{equation}
where $P_{\sigma^r}$ is the $P$-tableau of the reverse $\sigma^r$ of $\sigma$ and $P^t_{\sigma}$ is the transpose (reflection through anti-diagonal) of $P_{\sigma}$.

\begin{theorem}
\label{monotonpatt}
Consider $S_n(\tau)$ with $\tau=k(k-1)\cdots1.$ Then whenever $0<\epsilon<k-2$, we have 
\begin{equation}
  \label{eq.monoA}
\lim_{n\rightarrow\infty} \bP^{\tau}_n\left(\LIS_n>\frac{1+\epsilon}{k-1}\,n\right)^{1/n}\;=\; 
  \left[(1+\epsilon)^{(1+\epsilon)}\left(1-\frac{\epsilon}{k-2}\right)^{k-2-\epsilon}\right]^{-\frac{2}{k-1}}
\end{equation}
and
\begin{equation}
   \label{eq.monoB}
   \bP^{\tau}_n\left(\LIS_n<\frac{n}{k-1}\right)=0 \hspace{5mm}\hbox{for every $n$}.  
\end{equation}
Therefore
\[\frac{\LIS_n}{n}\;\to\;  \frac{1}{k-1} \hspace{5mm} \text{ in probability as } n\to \infty.\]
\end{theorem}

\begin{proof}[Proof of Theorem~\ref{monotonpatt}]

Fix the decreasing pattern $\tau=k(k-1)\cdots1.$ For any $\sigma \in S_n(\tau)$, there will be at most $k-1$ rows in the corresponding Ferrers shape because $\LDS_n(\sigma)\leq k-1$ and $\LDS_n(\sigma)=\LIS_n(\sigma^r)=$ number of columns in $P_{\sigma^r}=$ number of rows in $P_{\sigma}$ by  Equation~(\ref{Schensted}). Note also that it follows from the Erd\"{o}s-Szekeres lemma that $\LIS_n(\sigma)\geq \frac{n}{k-1}$ for all $\sigma\in S_n(\tau)$. This proves 
Equation~(\ref{eq.monoB}).

Let $\mathcal{H}_{n,k}$ denote the set of partitions of $n$ with at most $k-1$ rows in their
corresponding Ferrers shape. For notational convenience, we will consider 
\[   \mathcal{H}_{n,k}:= \{(\lambda_1,\cdots,\lambda_{k-1})\in \mathbb{Z}^{k-1}: \lambda_1\geq\cdots\geq\lambda_{k-1}\geq0, \sum_{i=1}^{k-1}\lambda_i=n\}.
\]
For a given $\epsilon>0,$ let $\mathcal{H}_{n,k}^{\epsilon}:=\{\lambda\in \mathcal{H}_{n,k}: \lambda_1>n\frac{1+\epsilon}{k-1}\}$ and let 
\[\mathcal{C}^{\epsilon}_k:=\left\{(x_1,\cdots,x_{k-1}) \in \mathbb{R}^{k-1} 
: \; x_1\geq x_2\geq \cdots \geq x_{k-1}\geq 0,\; \sum_{i=1}^{k-1}x_i=1, \; x_1\geq \frac{1+\epsilon}{k-1} \right\}.
\]

Let $\Lambda_n$ be a random variable taking values in $\mathcal{H}_{n,k}$ with the distribution $\mathcal{Q}\equiv\mathcal{Q}_{n,k}$ given by
\begin{equation}
   \label{eq.monoC}
\mathcal{Q}(\Lambda_n=\lambda):=\frac{d_{\lambda}^2}{|S_n(\tau)|}.
\end{equation}
This distribution can be considered as a ``shape distribution'' on the Ferrers shapes corresponding to the uniform distribution on permutations in $S_n(\tau).$ In particular, $\bP^{\tau}_n(\LIS_n(\sigma)=a)=\mathcal{Q}(\Lambda_{n,1}=a).$

Note that for $\lambda=(\lambda_1,\cdots,\lambda_{k-1})\in \mathcal{H}_{n,k}$, we have
\[\prod_{i=1}^{k-1}\lambda_i! \;\leq\; \prod_c h_c\;\leq\;  \prod_{i=1}^{k-1}(\lambda_i+k-1)!,\]
and hence
\begin{equation}
   \label{eq.monohook}
   \frac{n!}{\prod_{i=1}^{k-1}(\lambda_i+k-1)!} \;\leq \; \frac{n!}{\prod_c h_c} \;\leq \; 
      \frac{n!}{\prod_{i=1}^{k-1}\lambda_i!}.
\end{equation}
If $x\in \mathcal{C}^{\epsilon}_k$ and the sequence 
$\{\lambda(n)\}$      
(with $\lambda(n)\in \mathcal{H}^{\epsilon}_{n,k}$) satisfies 
$\lim_{n\rightarrow\infty}\lambda(n)/n \,=\,x$,
then by  Stirling's approximation, $n!\sim n^ne^{-n}\sqrt{2\pi n}$, we see 
from Equation (\ref{eq.monohook}) that
\begin{equation}
   \label{eq.monoE}
 \lim_{n\rightarrow\infty} d_{\lambda(n)}^{1/n} \;=\; D(x)  \; :=  \;
    \frac{1}{x_1^{x_1}\cdots x_{k-1}^{x_{k-1}}}
\end{equation}
where $0^0=1.$ Note that the maximum of $D(x)$ over $\mathcal{C}^{\epsilon}_k$ is attained 
at $x=\alpha$ defined by  $\alpha_1=\frac{1+\epsilon}{k-1}$ and 
$\alpha_2=\cdots=\alpha_{k-1}=\frac{1}{k-1}(1-\frac{\epsilon}{k-2})$, with the value $D(\alpha) \,=\,
(k-1)/[(1+\epsilon)^{(1+\epsilon)}(1-\frac{\epsilon}{k-2})^{k-2-\epsilon}]^{1/(k-1)}$.
(To see this, first note that if we fix $x_1=A$, then a Lagrange multiplier calculation
shows that the maximum is at $x(A):=(A,\frac{1-A}{k-2},\ldots,\frac{1-A}{k-2})$; then show
that $\log D(x(A))$ is a concave function of $A\in(0,1)$ with maximum at $A=\frac{1}{k-1}$.)

For $k\geq 2,$ Regev \cite{R} proved that the Stanley-Wilf limit of the increasing pattern $\tau=12\cdots k$ is given by
\begin{equation}
 \label{eq.Regev}
   L(12\cdots k)\;=\;(k-1)^2.
\end{equation}
Using Equations (\ref{eq.monoC}) and (\ref{eq.Regev}), we see that 
\begin{equation}
   \label{eq.monoF}
    \liminf_{n\rightarrow\infty} \mathcal{Q}(\Lambda_n\in \mathcal{H}_{n,k}^{\epsilon})^{1/n}
      \;\geq \;  \frac{D(\alpha)^2}{(k-1)^2}.
\end{equation}
For the complementary inequality, let $\hat \lambda(n):=\text{argmax}_{\lambda \in \mathcal{H}_{n,k}^{\epsilon}}\mathcal{Q}(\Lambda_n=\lambda).$ Then
\begin{equation}
   \label{eq.monoH}
   \mathcal{Q}(\Lambda_n\in \mathcal{H}_{n,k}^{\epsilon}) \;\leq \;
     |\mathcal{H}_{n,k}^{\epsilon}|\, \mathcal{Q}(\Lambda_n=\hat \lambda(n)).
 \end{equation}
Since $\mathcal{C}^{\epsilon}_k$ is a compact set and $\frac{\hat \lambda(n)}{n}\in \mathcal{C}^{\epsilon}_k$, there exists a subsequence $\hat \lambda(n_l)$  such that $\frac{\hat \lambda(n_l)}{n_l}$ converges to a 
point $z\in \mathcal{C}^{\epsilon}_k$ and 
\begin{equation}
   \label{eq.monoK}
   \lim_{l\to \infty}\mathcal{Q}(\Lambda_{n_l}=\hat \lambda(n_l))^{1/n_l}\;=\;\limsup_{n\to\infty}\mathcal{Q}(\Lambda_n=\hat \lambda(n))^{1/n}.
\end{equation}
Since $|\mathcal{H}_{n,k}^{\epsilon}|=O(n^{k-1})$ as $n\to \infty,$ we see from 
Equations (\ref{eq.monoH}) and (\ref{eq.monoK}) and the argument for Equation (\ref{eq.monoF})
that 
\begin{equation}
   \label{eq.monoG}
    \limsup_{n\rightarrow\infty} \mathcal{Q}(\Lambda_n\in \mathcal{H}_{n,k}^{\epsilon})^{1/n}
     \;\leq \;  \frac{D(z)^2}{(k-1)^2} \;\leq \;  \frac{D(\alpha)^2}{(k-1)^2}.
\end{equation}
Equations (\ref{eq.monoF}) and (\ref{eq.monoG}) imply that 
Equation (\ref{eq.monoA}) holds, and the theorem follows.
\end{proof} 

The next group of results concerns $\LIS_n$ when $\tau_1=k$.  The key observation
is that $\LIS_n$ is at least as big as the number of left-to-right maxima 
(Equation (\ref{eq.LISLR})).  Theorem  \ref{MainTh2} then tells us that $\LIS_n$ 
is at least a constant times $n$, with very high probability.

Let $\RL_n(\sigma)$ be the number of right-to-left maxima in the permutation $\sigma,$ and 
$\LR_n(\sigma)$ be the number of left-to-right maxima in $\sigma.$ Note that the left-to-right maxima (right-to-left maxima) in $\sigma$ form an increasing (decreasing) subsequence in $\sigma.$ In particular,
\begin{equation}
   \label{eq.LISLR}
   \LIS_n(\sigma)\geq \LR_n(\sigma)\hspace{5mm} \text{ and } \hspace{5mm}
     \LDS_n(\sigma)\geq \RL_n(\sigma).
\end{equation}     
Recall also the notation $S_n(\tau;\mathcal{P})$ from Section \ref{sec-not}.

\begin{theorem} 
\label{MainTh2}
Let $\tau=k\tau_2\cdots \tau_k \in S_k$. For every real $\delta$ such that $0<\delta<1/L(\tau)$, the following strict inequality holds:
\[\limsup_{n\to \infty}|S_n(\tau;\LR_n<\delta n)|^{1/n}<L(\tau).\]
\end{theorem}

Theorem~\ref{MainTh2} is an immediate consequence of Theorem \ref{GMainTh2}, which
is a stronger (and more technical) result.   We will also need Theorem \ref{GMainTh2} for the 
proof of Theorem \ref{regionthm}(b).  

For $\tau\in S_k$, $\delta>0$, and an interval $I$, define
\begin{align*}
S_n(\tau; &\, \LR_n[I]<\delta n)\\
&=\{\sigma \in S_n(\tau):\hbox{$\sigma$ has fewer than $\delta n$ left-to-right maxima $\sigma_i$ 
   such that $\sigma_i\in I$
} \}.
\end{align*}
\begin{theorem} 
\label{GMainTh2}
Let $\tau=k\tau_2\cdots \tau_k \in S_k$. 
For $0\leq \alpha_1<\alpha_2\leq1,$ let $I_n=[a_n,b_n]$ where $a_n/n\to\alpha_1$ and $b_n/n\to\alpha_2$. 
For every $0<\delta< \frac{\alpha_2-\alpha_1}{L(\tau)}$, the following strict inequality holds:
\[\limsup_{n\to \infty}|S_n(\tau;\LR_n[I_n]<\delta n)|^{1/n}<L(\tau).\]
\end{theorem}

Intuitively, this result says that for a typical member of $S_n(\tau)$, the
left-to-right maxima fill $[1,n]$ with a positive density throughout $[1,n]$.

We shall require an \textit{insertion operation} $\mathcal{I}$ which was introduced in \cite{AtM}. 
Suppose $\sigma \in S_n$ and $h \in [n].$ Let $J=\min\{j:\sigma_j\geq h\}.$ 
Define the permutation $\theta \in S_{n+1}$ as follows:
\begin{equation*}
\theta_i= \begin{cases}
\sigma_i &\text{ if } i<J\\
h &\text{ if } i=J\\
\sigma_{i-1}&\text{ if } i\geq J+1 \text{ and } \sigma_{i-1}<h\\
\sigma_{i-1}+1&\text{ if } i\geq J+1 \text{ and } \sigma_{i-1}\geq h\\
\end{cases}
\end{equation*}
We denote $\theta$ by $\mathcal{I}(\sigma;h).$   Think of $\mathcal{I}(\sigma;h)$ as
inserting a new left-to-right maximum in $\sigma$ at height $h$, while preserving all
the other left-to-right maxima of $\sigma$ (perhaps shifting them slightly).

The proof of Theorem~\ref{GMainTh2} is similar to the proof of Proposition 6.2 of \cite{AtM}.
The main idea is that repeated use of $\mathcal{I}$ can transform a permutation $\sigma$ 
in $S_n(\tau)$ into many other permutations  $\sigma'$
in $S_{n+r}(\tau)$ with $r$ more left-to-right maxima.  By insisting that $\sigma$ 
started with few left-to-right maxima, we significantly limit the number of choices of
$\sigma$ that could give rise to a particular $\sigma'$.  We then estimate how much smaller
the number of preimages  $\sigma$  is than the number of images $\sigma'$, which is
at most $L(\tau)^{n+r}$.

The following lemma says that if $\tau_1=k$, then $\mathcal{I}$ preserves $\tau$-avoidance.
For the proof, see the proof of Lemma $6.1$ in \cite{AtM}.
\begin{lemma}[\cite{AtM}]
    \label{lem.12}
Let $\tau=k\tau_2\cdots \tau_k \in S_k$. 
Assume $1\leq h_1<h_2<\cdots<h_r\leq n.$ Let $\sigma^{(r)}=\sigma\in S_n(\tau)$ and let 
$\sigma^{(l-1)}=\mathcal{I}(\sigma^{(l)};h_l)$ for $l=r,r{-}1,\cdots,2,1.$ 
We denote $\sigma^{(0)}$ by $\mathcal{I}(\sigma;h_1,h_2,\cdots,h_r).$ Then we have 
\[\sigma^{(0)}\in S_{n+r}(\tau).\] 
\end{lemma}
Observe also that $h_i+i-1$ is a left-to-right maximum of $\sigma^{(0)}$ for 
each $i\in [r]$.

\begin{proof}[Proof of Theorem \ref{GMainTh2}]

Assume $0<\delta<\frac{\alpha_2-\alpha_1}{L(\tau)}$ and choose a positive integer $M$ so that 
$\frac{L(\tau)}{M}<1-\frac{\delta L(\tau)}{\alpha_2-\alpha_1}.$ 
Let $r\equiv r_n:=\floor{\frac{b_n-a_n-2}{M}}$.

 For fixed $n$, define the \textit{intervals of heights} in $I_n$ as 
${\cal J}_i \,:=\,(\ceil{a_n}+(i-1)M,\ceil{a_n}+iM]$ for each $i\in [r].$ Then ${\cal J}_1,\cdots, {\cal J}_r$ are disjoint subintervals of $I_n.$

In the following, we shall use the notation of Lemma \ref{lem.12}.
Define the function $\Psi\equiv \Psi_n$ as
\[\Psi : S_n(\tau;\LR_n[I_n]<\delta n)\times [M]^r \to S_{n+r}(\tau)\]
such that $\Psi(\sigma,(\hat h_1,\hat h_2,\cdots,\hat h_r))=\mathcal{I}(\sigma;h_1,h_2,\cdots,h_r),$
where $h_i=\ceil{a_n}+\hat h_i +(i-1)M$ for $i\in [r]$ 
(so that $h_i\in {\cal J}_i$).
Corresponding to the intervals of heights ${\cal J}_i$, we define  the shifted
intervals of heights as follows:
\[   {\cal J}^{\Psi}_i\,:=\,{\cal J}_i+i-1  \;=\;(\ceil{a_n}+(i-1)(M+1),\ceil{a_n}+i(M+1)-1],
  \quad i=1,\cdots, r.\]
  Then ${\cal J}^{\Psi}_1.\ldots,{\cal J}^{\Psi}_r$ are disjoint subintervals of $(a_n,b_n+r)$.
  
Given $\sigma' \in$ Image $ \Psi$, we want to find an upper bound on the number of $(\bar \sigma,(\bar h_1,\bar h_2,\cdots,\bar h_r))$ in the domain of $\Psi$ such that 
\begin{equation}
   \label{eq.psi00}
   \Psi(\bar \sigma,(\bar h_1,\bar h_2,\cdots,\bar h_r))=\sigma' .
\end{equation}
For $i\in [r]$, let $b_i$ be the number of left-to-right maxima of $\sigma'$ in ${\cal J}^{\Psi}_i.$ 
Observe that if Equation (\ref{eq.psi00}) holds, then $\ceil{a_n}+\bar{h}_i+(i-1)M+i-1$ is a 
left-to-right maximum of $\sigma'$ in ${\cal J}^{\Psi}_i$.

Also note that there are at most $\delta n+r$ left-to-right 
maxima of $\sigma'$ in $(a_n,b_n+r)$. It follows that
\[
|\Psi^{-1}(\sigma')|\;\leq\; \displaystyle \prod_{i=1}^{r}b_i \;\leq\; 
 \left( \frac{\sum_{i=1}^rb_i}{r}\right)^r \;\leq\; \left(\frac{\delta n+r}{r}\right)^r.
\] 
Therefore
\[|S_n(\tau;\LR_n[I_n]<\delta n)|M^r\;=\sum_{\sigma'\in \text{Image } \Psi}|\Psi^{-1}(\sigma')|
 \, \; \leq \; \, |S_{n+r}(\tau)| \left(\frac{\delta n+r}{r}\right)^r,\]
and hence 
\[|S_n(\tau;\LR_n[I_n]<\delta n)| \, \, \leq \, \, L(\tau)^{n+r}\left(\frac{\delta n+r}{M r}\right)^r.\]
Since $\dfrac{r_n}{n} \to \dfrac{\alpha_2-\alpha_1}{M}$ as $n\to \infty$, it follows that 
\[\limsup_{n\to \infty}|S_n(\tau;\LR_n[I_n]<\delta n)|^{1/n} \;\leq\;  
 L(\tau)\left(\frac{\delta L(\tau)}{\alpha_2-\alpha_1}+\frac{L(\tau)}{M}\right)^{(\alpha_2-\alpha_1)/M}
   \;<\;  L(\tau).\]
Then Theorem~\ref{GMainTh2} follows.
\end{proof}

Theorem \ref{MainTh2} and Equation (\ref{eq.LISLR})  imply  the following.

\begin{corollary} 
\label{liscor}
Consider $S_n(\tau)$ with $\tau=k\tau_2\cdots \tau_k \in S_k.$ Then
\[\liminf_{n\to \infty}\frac{\E^{\tau}_n(\LIS_n)}{n}\geq \frac{1}{L(\tau)}.\]
\end{corollary}

\begin{proof}[Proof of Corollary~\ref{liscor}]
Recall from Equation (\ref{eq.LISLR}) 
 that $\LIS_n(\sigma)\geq \LR_n(\sigma)$ for every $\sigma\in S_n$. 
 Therefore for any $0<\delta<\frac{1}{L(\tau)}$,
\[\E^{\tau}_n(\LIS_n)\;\geq \;\E^{\tau}_n(\LR_n)\;\geq \; \E^{\tau}_n(\LR_n1_{\LR_n\geq \delta n})\; \geq\; 
    \delta n \, \frac{|S_n(\tau;\LR_n\geq \delta n)|}{|S_n(\tau)|} \,.\] 
Then the result follows from Theorem~\ref{MainTh2}.
\end{proof}

The next result provides a sharp contrast with Theorem \ref{MainTh2} in the case 
$\tau=k\tau_2\cdots\tau_k$. Thus it is of some independent interest, even though it tells
us nothing about longest increasing subsequences. 

\begin{theorem} 
\label{thm2413}
Assume $\tau \in S_k$ and satisfies at least one of the following conditions:
\begin{verse} 
(\textit{i}) $\tau_1<\tau_2$, or 
\\
 (\textit{ii})   $k$ occurs to the right of $k{-}1$ in $\tau$.
\end{verse} 
Then 
\begin{equation}
   \label{eq.liminfE}
  \liminf_{n\to \infty}\E^{\tau}_n(\LR_n)\leq L(\tau).
\end{equation}
If also $k\leq 5$ or if $\tau_1=1$ or if $\tau_k=k$, then we can replace the above ``lim inf'' by 
``lim sup.''
\end{theorem}
We shall use the following result in the proof of Theorem~\ref{thm2413}. 
For its proof, see Theorem 6.4, Remark 2, and Theorem 7.1 of \cite{AtM}.  Presumably the
result (\ref{ratiolim}) holds for every $\tau$, but the proof remains elusive.

\begin{theorem}[\cite{AtM}]
\label{AtMad} For every $\tau$ in $S_k$ with $k\leq 5$
or when $\tau_1$ (or $\tau_k$) equals 1 or $k$, we have 
\begin{eqnarray}
\label{ratiolim}
\lim_{n\to \infty}\frac{|S_{n+1}(\tau)|}{|S_n(\tau)|}=L(\tau) \,.
\end{eqnarray}
\end{theorem}
\begin{proof}[Proof of Theorem~\ref{thm2413}]
First assume that condition (\textit{ii}) holds, i.e.\ that $k$ occurs to the right of $k{-}1$ in $\tau$.

For each $\sigma\in S_n$ and $i=1,\cdots,n{+}1,$ let $\sigma^i$ be the permutation in $S_{n+1}$ obtained by inserting $n+1$ between the $(i{-}1)^{th}$ and $i^{th}$ positions of $\sigma.$ That is, $\sigma^i=\sigma_1\cdots\sigma_{i-1}(n+1)\sigma_i\cdots\sigma_n.$ 

For a given $\sigma\in S_n(\tau)$, define
\[\W_n(\sigma)=\text{ the number of values of }i\text{ for which } \sigma^i \in S_{n+1}(\tau).\]
Let $\sigma \in S_n(\tau)$ and suppose $\sigma_i$ is a left-to-right maximum. 
Suppose that $\sigma^i$ contains a subsequence that forms the pattern $\tau$. 
Then the following must be true:

\begin{verse}
$\bullet$ For some $j<i$, $\sigma_j$ and $n+1$ must be the second-largest and  largest 
entries, respectively, in this subsequence.

$\bullet$ This subsequence cannot include $\sigma_i$, because $\sigma_i$ is a left-to-right maximum 
and thus $\sigma_j< \sigma_i$.
\end{verse}

But then
replacing $n+1$ with $\sigma_i$ in this subsequence of $\sigma^i$ would produce a 
subsequence of $\sigma$ that forms
the pattern $\tau$, which is a contradiction.
Hence $\sigma^i \in S_{n+1}(\tau).$ 

By the above argument, for any $\sigma \in S_n(\tau)$ we have 
\begin{equation}
\label{RLW}
\LR_n(\sigma)\leq \W_n(\sigma).
\end{equation}

As argued in Lemma 2.1 of \cite{ML}, we note that for each $\rho \in S_{n+1}(\tau)$, there is a unique $\sigma \in S_n(\tau)$ and a unique $i\in [n+1]$ such that $\sigma^i=\rho$. It follows that
\begin{equation*}
\sum_{\sigma \in S_n(\tau)}\W_n(\sigma)=S_{n+1}(\tau).
\end{equation*}
Dividing both sides by $|S_n(\tau)|$ gives us 
\begin{equation}
   \label{eq.EWSS}
     \E^{\tau}_n(\W_n)=\frac{|S_{n+1}(\tau)|}{|S_n(\tau)|}.
\end{equation}
From \cite{A}, we also have 
\begin{equation}
    \label{eq.Ltau}
L(\tau):=\lim_{n\to \infty}|S_n(\tau)|^{1/n}=\sup_{n\geq 1}|S_n(\tau)|^{1/n}.
\end{equation}
Hence  $\liminf_{n\rightarrow\infty}\E^{\tau}_n(\W_n)\leq L(\tau)$  by Equations 
(\ref{eq.EWSS}) and  (\ref{eq.Ltau}).
Together with Equation (\ref{RLW}), this proves Equation (\ref{eq.liminfE}). 
The final statement of Theorem \ref{thm2413} follows from Theorem~\ref{AtMad}. 

This concludes the proof of Theorem \ref{thm2413} under condition (\textit{ii}).
The theorem under condition (\textit{i}) will follow upon applying the bijection 
$B:\sigma\mapsto (\sigma^{-1})^{rc}$, which corresponds to reflection through the decreasing
diagonal.  We only require two observations:  
firstly,  that $\tau$ satisfies condition (\textit{i}) if and only if $B(\tau)$ satisfies condition (\textit{ii}),
and secondly that $\LR_n(\sigma) \,=\, \LR_n(B(\sigma))$ for every $\sigma\in S_n$ 
(indeed, $j=\sigma_i$ is a left-to-right maximum of $\sigma$ if and only if 
$(B(\sigma))_{n+1-j}=n+1-i$ is a left-to-right maximum of $B(\sigma)$).

In particular, for $k=4$, we conclude that $\E_n^{\tau}(\LR_n)$ is of order $n$ when $\tau_1=4$,
and is bounded for every other $\tau \in S_4$ except perhaps $2143$.
\end{proof}

For all nontrivial cases of pattern-avoidance that we have considered, the longest 
monotone subsequence (i.e., the  maximum of $\LIS_n$ and $\LDS_n$) is of order $n$ on average,
which contrasts with the order $\sqrt{n}$ that holds for the set of all permutations.
We don't yet know whether this order $n$ behaviour holds in general under pattern avoidance,
but this hypothesis is consistent with our limited simulation experiments so far. 
For example, Figure~\ref{table2413} summarizes some simulation results on the length 
of the longest increasing subsequence for random permutations in 
$S_n(2413)$. 
The data suggest that $\E^{2413}_n(\LIS_n)/n$ converges to a number between 0.2 and 0.25.

\begin{figure}[h!]
\begin{subfigure}{0.54\textwidth}
\begin{tabular}{ |c|c|c|}
\hline
$n$ & $\E^{2413}_n(\LIS_n)$ & $\E^{2413}_n(\LIS_n)/n$ \\
\hline
\hline
75& $22.3425\pm 0.3783$ & $0.2979\pm0.0050$ \\
\hline 
100& $29.1375\pm 0.5336$ & $0.2914\pm0.0053$ \\
\hline 
125& $35.0050\pm 0.5866$ & $0.2800\pm0.0047$ \\
\hline 
150& $39.9825\pm 0.7321$ & $0.2665\pm0.0049$ \\
\hline 
175& $46.4550\pm 0.8172$ & $0.2655\pm0.0047$ \\
\hline 
200& $51.5350\pm 0.8890$ & $0.2577\pm0.0044$ \\
\hline 
225& $56.3400\pm 1.0091$ & $0.2504\pm0.0045$ \\
\hline 
235& $59.3122\pm 1.0020$ & $0.2524\pm0.0043$ \\
\hline 
250& $62.9425\pm 1.1173$ & $0.2518\pm0.0045$ \\
\hline 
\end{tabular}
\end{subfigure}
\begin{subfigure}{0.46\textwidth}
\includegraphics[width=8cm, height=6.35cm]{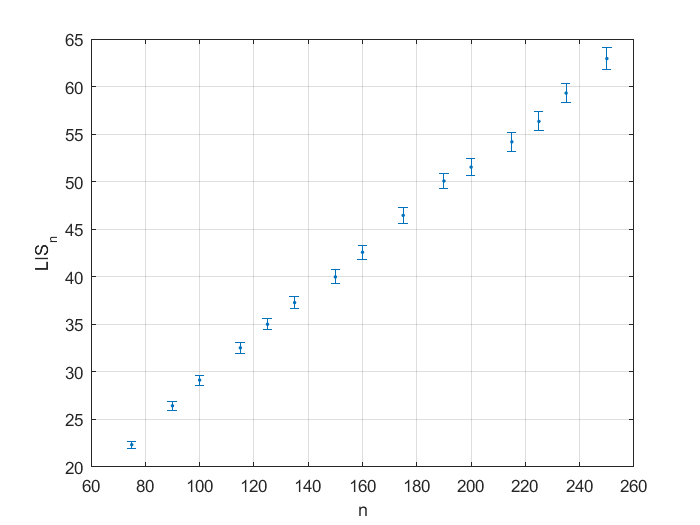}
\end{subfigure}
\caption{$95\%$ confidence interval on $\E^{2413}_n(\LIS_n)$ and $\E^{2413}_n(\LIS_n)/n$ for permutations avoiding
2413. For each $n$, we used a sample of  $400$ (approximately independent) permutations
generated by the Markov Chain Monte Carlo method of \cite{ML}. }
\label{table2413} 
\end{figure}

\begin{figure}[h!]
\begin{subfigure}{0.55\textwidth}
\includegraphics[width=7cm, height=7cm]{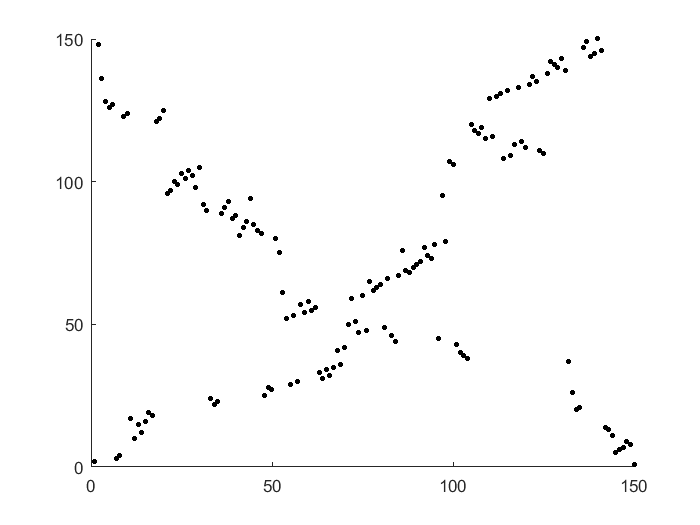}
\end{subfigure}
\begin{subfigure}{0.55\textwidth}
\includegraphics[width=7cm,height=7cm]{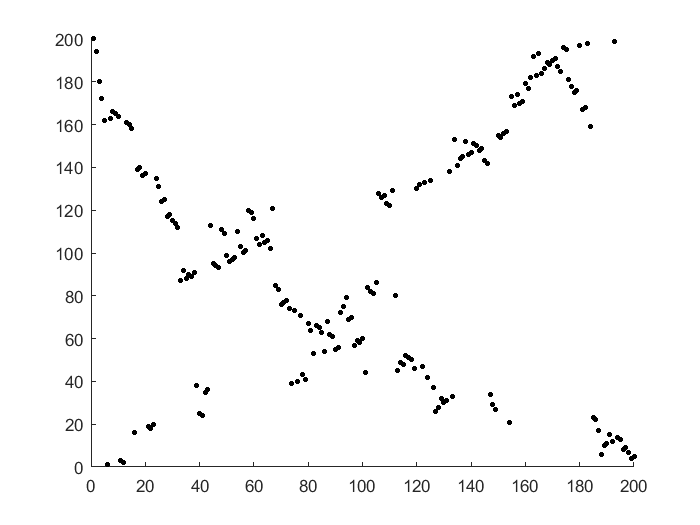} 
\end{subfigure}
\caption{Examples of randomly generated 2413-avoiding permutations in $S_{150}(2413)$ and 
$S_{200}(2413)$. These are reminiscent of plots given in \cite{BBFGP} of separable permutations, 
$S_n(2413,3142)$,  which appears to be a more tractable class.}
\label{fig2413}
\end{figure}

As described in Section \ref{LASsection}, it is known that the expected length of the longest
alternating subsequence is either bounded or else asymptotically proportional to $n/2$,
 for all cases of $S_n(T)$ with $T$ consisting of one or two patterns of length 3.  
 We do not know whether this extends to all longer patterns, but we can prove the following.
 
 \begin{theorem}
    \label{thm.alt4}
Consider $S_n(\tau)$ with $\tau\in S_k$ where $k\geq 4.$ Assume  that either
\begin{itemize}
\item[(a)]  $\tau_1=k$ and $\tau_2\neq k-1$, or 
\item[(b)]  $|\tau_i-\tau_{i+1}|>1$ for every $i\in[k-1]$. 
\end{itemize}
Then 
\[      \liminf_{n\rightarrow\infty} \frac{\E^{\tau}_n(\LAS_n)}{n}   \;\geq \;  \frac{2}{L(\tau)^2} \,.   \]
\end{theorem}
In particular, Theorem \ref{thm.alt4} proves that $\E^{\tau}_n(\LAS_n)/n$ is bounded away from 0 
for the patterns $\tau=4231$ (part (a)) and $\tau=2413$ (part (b)). 
Permutations which satisfy the condition in part (b) are called $1$-\textit{prolific} in \cite{BHT}.

The strategy of the proof is similar to that of Theorem \ref{GMainTh2}.

\begin{proof}[Proof of Theorem \ref{thm.alt4}] $(a)$
Assume $\tau\in S_k$ with $\tau_1=k$ and $\tau_2\neq k-1$.
It suffices to prove that for every positive real $\delta<L(\tau)^{-2}$, we have 
\begin{equation}
   \label{eq.LASdelta}
       \limsup_{n\rightarrow\infty}|S_n(\tau;\LAS_n<2\delta n)|^{1/n}   \,<\,L(\tau). 
\end{equation}

Suppose $\sigma \in S_n$ and $h \in [n].$ 
Let $J=\min\{j:\sigma_j\geq h\}.$ Define the permutation $\theta \in S_{n+2}$ as follows:
\begin{equation*}
\theta_i= \begin{cases}
\sigma_i &\text{ if } i<J\\
h+1 &\text{ if } i=J\\
h & \text{ if }i=J+1 \\
\sigma_{i-2}&\text{ if } i\geq J+2 \text{ and } \sigma_{i-2}<h\\
\sigma_{i-2}+2&\text{ if } i\geq J+2 \text{ and } \sigma_{i-2}\geq h\\
\end{cases}
\end{equation*}
We denote $\theta$ by $\mathcal{I}_2(\sigma;h).$ 
Think of  $\mathcal{I}_2(\sigma;h)$  as inserting two new points into $\sigma$ in a 21 pattern,
adjacent in location as well as in height, with the left point being  
a new left-to-right maximum at height $h+1$.

Assume $1\leq h_1<h_2<\cdots<h_r\leq n.$ Let $\sigma^{(r)}=\sigma$ and let 
$\sigma^{(l-1)}=\mathcal{I}_2(\sigma^{(l)};h_l)$ for $l=r,r-1,\cdots,1.$ 
Then we denote $\sigma^{(0)}$ by 
$\mathcal{I}_2(\sigma;h_1,h_2,\cdots,h_r).$ 
It is not hard to see that $\sigma^{(0)}$ avoids $\tau$, and that $\sigma^{(0)}$
has left-to-right maxima with heights $h_i+2i-1$ ($i\in [r]$).

Fix $0<\delta<L(\tau)^{-2}$ and choose a positive integer $M$ such that $\delta+\frac{1}{M} < L(\tau)^{-2}$.
Let $r\equiv r_n:=  \lfloor n/M\rfloor$.

For each $i\in [r],$ we define 
${\cal J}_i=((i-1)M,iM]$ (the \textit{``intervals of heights''}).

For each permutation $\sigma\in S_n$, define the set of heights
\[   {\cal A}^*(\sigma)  \;:=\;  \{\sigma_t:  \hbox{ $\sigma_t$ is a left-to-right maximum and }
   \sigma_{t+1}=\sigma_t-1\}\,.  \]
E.g., ${\cal A}^*(254367981)\,=\,\{5,9\}$.
Then it follows that  $\LAS_n(\sigma)  \geq 2|{\cal A}^*(\sigma)|$.  
Next, let 
\[   S_n^*(\tau)  \;:=\;  \{\sigma\in S_n(\tau):   |{\cal A}^*(\sigma)|<\delta n\}\,.   \]
Then we have
\begin{equation}
  \label{eq.SsubS}
     S_n(\tau;\LAS_n<2\delta n)  \;\subset \;  S^*_n(\tau)   \,. 
\end{equation}

Define the function $\Psi_2\equiv \Psi_{2,n}$ as
\[\Psi_2 : S_n^*(\tau)\times [M]^r \to S_{n+2r}(\tau)\]
such that 
$\Psi_2(\sigma,(\hat h_1,\hat h_2,\cdots,\hat h_r))\,=\,\mathcal{I}_2(\sigma;h_1,h_2,\cdots,h_r),$
where $h_i=\hat h_i +(i-1)M$ for $i\in[r]$ (so that $h_i\in {\cal J}_{i}$).
We also  define  the set of shifted intervals of heights as follows:
\[  
     {\cal J}^{\Psi}_i  \;:=\; {\cal J}_i+2i-1  \;=\; ((i-1)(M+2)+1,i(M+2)-1], \quad i=1,\cdots, r.\]
Then ${\cal J}^{\Psi}_i,\ldots,{\cal J}^{\Psi}_r$ are disjoint intervals in $[1,n+2r]$.

Given $\sigma' \in$ Image $ \Psi_2$, we would like to find an upper bound on the number of 
$(\bar \sigma,(\bar h_1,\bar h_2,\cdots,\bar h_r))$ in the domain of $\Psi_2$ such that 
\begin{equation}
   \label{eq.psisigbar}
   \Psi_2(\bar \sigma,(\bar h_1,\bar h_2,\cdots,\bar h_r))=\sigma' .
\end{equation}
For $i\in[r]$, let $b^*_i=|{\cal A}^*(\sigma')\cap {\cal J}^{\Psi}_i|$.
Observe that if Equation (\ref{eq.psisigbar}) holds, then 
$\bar{h}_i+(i-1)M+(2i-1)\in {\cal A}^*(\sigma')\cap {\cal J}^{\Psi}_i$. 
Note that $|{\cal A}^*(\sigma')|\leq \delta n+r$.
Therefore 
\[|\Psi_2^{-1}(\sigma')| \;\leq \;  \displaystyle \prod_{i=1}^{r}b^*_i  \;\leq \;  
   \left( \frac{\sum_{i=1}^rb^*_i}{r}\right)^r  \;\leq\; \left(\frac{\delta n+r}{r}\right)^r.\] 
Therefore, 
\[|S_n^*(\tau)| \, M^r \;=\; \sum_{\sigma'\in \text{Image } \Psi_2}|\Psi_2^{-1}(\sigma')|\; \leq \; 
   |S_{n+2r}(\tau)| \left(\frac{\delta n+r}{r}\right)^r,\]
and hence 
\[|S_n^*(\tau)| \, \leq \, \, L(\tau)^{n+2r}\left(\frac{\delta n+r}{Mr}\right)^r.\]
It follows  that 
\[\limsup_{n\to \infty}|S^*_n(\tau)|^{1/n} \; \leq \;   L(\tau)\left(L(\tau)^2
   \left(\delta +\frac{1}{M}\right)\right)^{1/M}\;<\;L(\tau).\]
Equation (\ref{eq.LASdelta}) follows from this and Equation (\ref{eq.SsubS}), and the 
result of part (a) is proved.

\medskip
\noindent

$(b)$ Assume $\tau\in S_k$ and $|\tau_i-\tau_{i-1}|>1$ for every $i\in[k-1]$.
Given $\sigma\in S_n(\tau)$ and $h\in[n]$, let ${\cal H}^*(\sigma;h)$ be the permutation $\theta$
defined as follows.  Let $J$ be the index such that $\sigma_J=h$ (note the distinction from 
part (a) here), and let
\begin{equation*}
\theta_i= \begin{cases}
\sigma_i &\text{ if } i\leq J \text{ and }\sigma_i\leq h\\
\sigma_i+2 &\text{ if } i\leq J \text{ and }\sigma_i > h\\
h+2 &\text{ if } i=J+1\\
h+1 & \text{ if }i=J+2 \\
\sigma_{i-2}&\text{ if } i > J+2 \text{ and } \sigma_{i-2}\leq h\\
\sigma_{i-2}+2&\text{ if } i > J+2 \text{ and } \sigma_{i-2}> h
\end{cases}
\end{equation*}
Then $\theta\in S_{n+2}(\tau)$.   In effect, ${\cal H}^*$ inserts two points into $\sigma$ to the right of
location $J$, so that $\sigma_J$ ($=\theta_J$) and the two new points 
($\theta_{J+1}$ and $\theta_{J+2}$) form a 132 pattern with three contiguous heights.
As in part (a), assume $1\leq h_1<h_2<\cdots<h_r\leq n.$ Let $\sigma^{(r)}=\sigma$ and let 
$\sigma^{(l-1)}=\mathcal{H}^*(\sigma^{(l)};h_l)$ for $l=r,r-1,\cdots,1.$ 
Then we denote $\sigma^{(0)}$ by $\mathcal{H}^*(\sigma;h_1,h_2,\cdots,h_r).$ 

For each permutation $\sigma$, we define the set of heights
\[    {\cal B}^*(\sigma) \;:=\; \{\sigma_t \,:\, \sigma_t=\sigma_{t+1}-2=\sigma_{t+2}-1\} \,.  \]
Then $\LAS_n(\sigma)\,\geq \,2|{\cal B}^*(\sigma)|$.  Arguing as in part (a) with the same
choice of $\delta$, $M$, and $r$, we obtain
\[    |S_n(\tau \,; \, |{\cal B}^*(\sigma)|<\delta n)|  \;\leq \;  \,L(\tau)^{n+2r}
     \left(\frac{\delta n+r}{Mr}\right)^r   \, \]
and the result of part (b) follows.
\end{proof}

\subsection{Rare regions for $S_n(\tau)$.}
\label{rsection}

As noted in Section \ref{sec-intrare}, plots of random permutations avoiding some patterns
have large regions that are usually empty (see Figure~\ref{regionfig}).
The results of this subsection relate to these regions. 

Without loss of generality, we shall assume for the rest of this subsection
that $\tau \in S_k$ and $\tau_1>\tau_k$.
Recall from Section \ref{sec-intrare} that the ``rare region'' $\mathcal{R}\equiv\mathcal{R}(\tau)$ 
is the set
\begin{align*}
\mathcal{R}=\{(x,y)\in [0,1]^2: \text{ for all sequences } \{(I_n,J_n)\}_{n\geq1}\text{ such that }(I_n,J_n)\in[n]^2\text{ and }\\ \lim_{n\to\infty}\left(\frac{I_n}{n},\frac{J_n}{n}\right)=(x,y), 
\text{ we have }\limsup_{n\to\infty}|S_n(\tau;\sigma_{I_n}=J_n)|^{1/n}<L(\tau) \},
\end{align*}
We also defined  $\mathcal{G}=[0,1]^2\setminus \mathcal{R}$ (the ``good region'') and 
$\mathcal{R}^{\uparrow}=\mathcal{R}\cap \{(x,y)\in [0,1]^2:y>x\}.$
For every $x\in [0,1]$, let
\[   r^{\uparrow}(x)  \,=\, \sup\{ y : (x,y)\in \mathcal{G}\}   \hspace{5mm}\hbox{and}\hspace{5mm}
   r^{\downarrow}(x)  \,=\, \inf\{ y : (x,y)\in \mathcal{G}\} \,.
\]
(We shall see that $r^{\uparrow}$ and $r^{\downarrow}$ are well-defined functions since
the set $\{ y : (x,y)\in \mathcal{G}\}$ is never empty.)

By Theorem \ref{thm.atm}, we know
that $\mathcal{R}^{\uparrow}\neq \emptyset$ 
if and only if $\tau_1=k$; 
in particular, when $\tau_1\neq k$, then $r^{\uparrow}$ is identically 1.
(Similarly, $r^{\downarrow}$ is identically 0 when $\tau_k\neq 1$.)
The other case, in which $\tau_1=k$, is addressed in the following theorem.
\begin{theorem}
\label{regionthm}
Assume $\tau=k\tau_2\cdots \tau_k \in S_k$.   
\begin{itemize}
\item[(a)]If $(x,y)\in \mathcal{G},$ then the convex hull of $\{(x,y),(0,0),(1,1)\}$ is contained in $\mathcal{G}.$  In particular, $\mathcal{G}$ contains the diagonal $\{(x,x):x\in [0,1]\}$.
\item[(b)]  
 $\mathcal{G} \,\subset \,  \{(x,y)\in[0,1]^2:  \,y\leq L(\tau)x \hbox{ and } y\leq 1-(1-x)/L(\tau)\}$.
\item[(c)]  The function $r^{\uparrow}$ satisfies $r^{\uparrow}(0)=0$, $r^{\uparrow}(1)=1$, and 
$r^{\uparrow}(x)\geq x$ for every $x\in (0,1)$.
\item[(d)]  The function $r^{\uparrow}$ is strictly increasing and Lipschitz continuous, 
with left and right derivatives contained in the interval $[L(\tau)^{-1},L(\tau)]$ at every point.
\end{itemize}
\end{theorem}

Theorem \ref{regionthm}(b) implies that for every $\tau\in S_k$ with $\tau_1=k$, the closure of  
$\mathcal{R}^{\uparrow}(\tau)$ includes the point $(0,0).$  
This had been proven in Proposition 9.2 of \cite{AtM} under an additional assumption (called
``TPIP'') on $\tau_2\cdots \tau_k$. Theorem \ref{regionthm}(b) shows that
this additional assumption is unnecessary.

\begin{proof}[Proof of Theorem \ref{regionthm}] 
\textit{(a)} Note that $\{(0,0),(1,1)\}\subset \mathcal{G}$ because $|S_n(\tau;\sigma_1=1)|=|S_{n-1}(\tau)|$ and $|S_n(\tau;\sigma_n=n)|=|S_{n-1}(\tau)|$ for all $n\geq 1.$ 

Let $(x,y)\in\mathcal{G}.$ Then there exist $I_n, J_n\in[n]$ such that $(\frac{I_n}{n},\frac{J_n}{n})\to (x,y)$ and 
\[\limsup_{n\to\infty}|S_n(\tau;\sigma_{I_n}=J_n)|^{\frac{1}{n}} \;=\; L(\tau).\]
Let $t\in(0,1)$ and let $m:=m_n$ be a sequence of integers such that $\frac{n}{n+m}\to t$. Then it follows that $(\frac{I_n}{n+m},\frac{J_n}{n+m})\to (tx,ty)$ and 
\begin{align*}
|S_{n+m}(\tau;\sigma_{I_n}=J_n)|^{\frac{1}{n+m}} \;\geq\; 
  |S_n(\tau;\sigma_{I_n}=J_n)|^{\frac{1}{n}\frac{n}{n+m}}|S_m(\tau)|^{\frac{1}{m}\frac{m}{m+n}}
\end{align*}
(the inequality is proved using the injection 
$S_n(\tau)\times S_m(\tau)\hookrightarrow S_{n+m}(\tau)$ given by 
the direct sum (defined in Section \ref{sec-not})).

Hence we have
\[\limsup_{n\to \infty}|S_{n+m}(\tau;\sigma_{I_n}=J_n)|^{\frac{1}{n+m}}  \;\geq \; 
  L(\tau)^{t}L(\tau)^{1-t} \;=\; L(\tau).\]
Therefore $\mathit{l}_1:=\{(0,0)+t(x,y): 0<t<1\}\subset\mathcal{G}.$ By a similar argument, we 
can show that $\mathit{l}_2:=\{(1-t)(1,1)+t(x,y): 0<t<1\}\subset\mathcal{G}$ and 
$\mathit{l}_3:=\{(0,0)+t(1,1): 0<t<1\}\subset\mathcal{G}.$ Since any point in the 
convex hull of $\{(x,y),(0,0),(1,1)\}$ can be written as a linear combination of $(0,0)$ and 
a point from $\mathit{l}_2$, the result follows.

\medskip
\noindent
\textit{(b)} Assume $(x,y)\in[0,1]^2$ and $y>L(\tau)x$.
Assume that $\left(\frac{I_n}{n},\frac{J_n}{n}\right)\to(x,y)$.  
If $\sigma_I=J,$ then every left-to-right maxima $\sigma_i$ with $\sigma_i<J$ must satisfy $i<I.$ Therefore, 
\[S_n(\tau;\sigma_{I_n}=J_n) \;\subset\; S_n(\tau;\LR_n([1,J_n])<I_n).\]
In Theorem~\ref{GMainTh2}, let $\alpha_1=0,\alpha_2=y,$ and choose $\delta$ such that $x<\delta<\frac{y}{L(\tau)}.$ Then $S_n(\tau;\sigma_{I_n}=J_n)\subset S_n(\tau;\LR_n([1,J_n])<\delta n)$ for all sufficiently large $n$. 
Therefore, by Theorem \ref{GMainTh2}, we see that $(x,y)\in\mathcal{R}$.  Therefore  
$\mathcal{G}\subset \{(x,y)\in [0,1]^2:y\leq L(\tau)x\}$.  

To show the rest of (b),
consider the mapping $(x,y)\mapsto (1{-}y,1{-}x)$, which is the reflection through the line 
$x+y=1$.  This reflection corresponds to the bijection $\sigma\mapsto (\sigma^{-1})^{rc}$.
Letting $\tilde{\tau} = (\tau^{-1})^{rc}$, we have $L(\tilde{\tau})=L(\tau)$ and $\tilde{\tau}_1=k$.
Now suppose  $(x,y)\in\mathcal{G}(\tau)$.   Then $(1-y,1-x)\in \mathcal{G}(\tilde{\tau})$, and hence
$(1-x) \leq L(\tilde{\tau})(1-y)$ by the preceding paragraph.  Therefore $y\leq 1-(1-x)/L(\tau)$.  
This completes the proof of (b).

\medskip
\noindent
\textit{(c)} Since $(0,0)\in \mathcal{G}$, part (b) implies that $r^{\uparrow}(0)=0$.  The rest of
part (c) follows from part (a).
 
\medskip
\noindent
\textit{(d)} Strict monotonicity follows from the convex hull property of part (a), together with the fact 
(a consequence of (b)) that $r^{\uparrow}(x)<1$ for every $x\in(0,1)$. 
Next, observe that if $0<u<v<1$, then the convex hull property implies that the point 
$(u,r^{\uparrow}(u))$ cannot be below the line segment joining $(0,0)$ to 
$(v,r^{\uparrow}(v))$.  That is, $r^{\uparrow}(v)/v \,\leq \, r^{\uparrow}(u)/u$, and so 
\[   \frac{ r^{\uparrow}(v)-r^{\uparrow}(u)}{v-u}  \;\leq \;  
     \frac{ r^{\uparrow}(u)\left(\frac{v}{u}\right)-r^{\uparrow}(u)}{v-u}  \;=\;  \frac{r^{\uparrow}(u)}{u}
       \;\leq \;  L(\tau)  
\]
(where the final inequality follows from part (b)).  This proves the upper bound on the 
derivatives.  The lower bound follows using  the reflection argument from part (b). 
\end{proof}

Theorem \ref{thm.boundary} below holds for any pattern $\tau$.  Together with 
Theorem \ref{regionthm}, it shows that
the  graph of $r^{\uparrow}$ is precisely the boundary 
between $\mathcal{R}^{\uparrow}$ and 
$\mathcal{G}$ (when $\mathcal{R}^{\uparrow}$ is not empty).

\begin{theorem}
   \label{thm.boundary}
Let $\tau\in S_k$.  Then $\mathcal{G} \,=\, \{(x,y)\in [0,1]^2:  r^{\downarrow}(x)\leq y \leq 
   r^{\uparrow}(x)\}$. 
\end{theorem}

\begin{proof}[Proof of Theorem \ref{thm.boundary}]

Without loss of generality, assume $\tau_1>\tau_k$.  It suffices to show that 
for every $(x,y)\in[0,1]^2$ with $y\geq x$, we have $(x,y)\in\mathcal{G}$ if and only if 
$y\leq r^{\uparrow}(x)$. We already know this from the discussion preceding 
Theorem \ref{regionthm} when $\tau_1\neq k$, so assume $\tau_1=k$.
The desired result will follow from the convex hull property of Theorem \ref{regionthm}(a)
if we can prove that the point $(x,r^{\uparrow}(x))$ is in $\mathcal{G}$ for every
$x\in (0,1)$.  We shall accomplish this by proving that $\mathcal{G}$ is closed.

Proving that $\mathcal{G}$ is closed is essentially an exercise in analysis.  Here are the details.
Consider a point $(x,y)$ in the
closure of $\mathcal{G}$.  It suffices to construct a strictly increasing  sequence of 
natural numbers $n(1),n(2),\ldots$ and a sequence of pairs of integers 
$\{(i_{n(m)},j_{n(m)}):m\geq 1\}$ such that $(i_{n(m)},j_{n(m)})\in [n(m)]^2$ for every $m$,
$\lim_{m\rightarrow\infty}(i_{n(m)},j_{n(m)})/n(m) \,=\,(x,y)$, and 
$\lim_{m\rightarrow\infty}|S_{n(m)}(\tau;\sigma_{i_{n(m)}}=j_{n(m)})|^{1/n(m)} \,=\,L(\tau)$.
We shall construct the sequences inductively.  Let $n(1)=1=i_1=j_1$. Given $m>1$ and
$n(m-1)\in \mathbb{N}$, choose $(x(m),y(m))\in \mathcal{G}$ such that 
$||(x(m),y(m))-(x,y)||_1\,<\,1/m$. Then there exists a sequence $\{(I_n(m),J_n(m)):n\geq 1\}$
such that $(I_n(m),J_n(m))\in[n]^2$ for every $n$, 
\begin{align*}
  \lim_{n\rightarrow\infty}&\left(\frac{I_n(m)}{n},\frac{J_n(m)}{n}\right)\,=\,(x(m),y(m)),
\hspace{5mm}\hbox{and}\hspace{5mm}  \\
& \limsup_{n\rightarrow\infty}|S_n(\tau;\sigma_{I_n(m)}=J_n(m))|^{1/n} \,=\,L(\tau).
\end{align*}
Therefore we can choose $n'>n(m-1)$ such that 
\begin{align*}
 |S_{n'}&(\tau;\sigma_{I_{n'}(m)}=J_{n'}(m))|^{1/n'} \,\geq\,L(\tau)-\frac{1}{m} 
  \hspace{5mm}\hbox{and}\hspace{5mm} \\
  & \left\|  \left(\frac{I_{n'}(m)}{n'},\frac{J_{n'}(m)}{n'}\right)-(x(m),y(m))\right\|_1 \,<\,\frac{1}{m} \,. 
\end{align*}
Let $n(m)=n'$, $i_{n(m)}=I_{n'}(m)$, and $j_{n(m)}=J_{n'}(m)$.
One can now check that these sequences have the desired properties.  Hence
$(x,y)\in \mathcal{G}$.  Thus $\mathcal{G}$ is closed, and the theorem follows.
\end{proof} 

It is known that $r^{\uparrow}$ is the identity function $r^{\uparrow}(x)=x$
for the monotone pattern $\tau=k(k-1)\cdots 1$, as well as for some other patterns
\cite{MP}. So far, there is 
no pattern $\tau\in S_k$ with $\tau_1=k$ for which we can prove that 
$r^{\uparrow}$ is not the identity function.  Simulations are not yet clear
about whether $4231$ is one such pattern. See Figure~\ref{regionfig} as well as
\cite{Bevan1}. We conjecture that $r^{\uparrow}$ is always a concave function.

\section{Acknowledgments} 
The authors wish to thank a referee for many helpful comments which have significantly 
improved the presentation of the paper.

\end{document}